\documentclass[graybox]{svmult} 
\usepackage{helvet}         
\usepackage{courier}        
\usepackage{type1cm}        
\usepackage{makeidx}         
\usepackage{graphicx}        
\usepackage{multicol}        
\usepackage[bottom]{footmisc}
\usepackage{bm}
\usepackage{mathtools}

\makeindex             

\usepackage{enumitem}
\usepackage{latexsym}
\usepackage{dsfont}
\usepackage{bbm}
\usepackage{amssymb}
\usepackage{amsmath}
\usepackage{mathdots}
\usepackage{graphicx}
\usepackage{bm}
\usepackage{type1cm}
\usepackage{latexsym}
\usepackage{mathrsfs}
\usepackage{dsfont}
\usepackage{bbm}
\usepackage{soul}
\usepackage{epic}
\usepackage{tikz}
\usetikzlibrary{automata, positioning}
\usepackage{enumitem}
\usepackage[normalem]{ulem}
\usepackage{xcolor}

\newtheorem{Theorem}{Theorem}[section]
\newtheorem{Lemma}[Theorem]{Lemma}

\newtheorem{Prop}[Theorem]{Proposition}
\newtheorem{Rem}[Theorem]{Remark}

\def\cC{\mathscr{C}}

\def\cF{\mathscr{F}}
\def\cG{\mathscr{G}}

\def\cR{\mathscr{R}}
\def\cS{\mathscr{S}}

\def\fA{\mathfrak{A}}

\def\bbD{\mathbb{D}}
\def\Erw{\mathbb{E}}

\def\H{\mathbb{H}}

\def\N{\mathbb{N}}

\def\Prob{\mathbb{P}}
\def\P{\mathbb{P}}

\def\R{\mathbb{R}}     
\def\bbS{\mathbb{S}}

\def\U{\mathbb{U}}

\def\Z{\mathbb{Z}}

\def\eps{\varepsilon}

\def\vth{\vartheta}

\def\1{\mathbf{1}}
\def\3{{\ss}}
\def\llam{\lambda\hspace{-5.1pt}\lambda}

\def\eqdist{\stackrel{d}{=}}

\def\wh{\widehat}

\def\sign{\textsl{sign}}

\def\diag{\textsl{diag}}

\def\IRg{\R_{\scriptscriptstyle >}}
\def\IRge{\R_{\scriptscriptstyle\geqslant}}
\def\IRl{\R_{\scriptscriptstyle <}}
\def\IRle{\R_{\scriptscriptstyle\leqslant}}

\newcommand{\Lip}{\textsf{Lip}\hspace{1pt}}

\newcommand{\ps}[1]{\hspace{-1.2pt}{}^{{\scalebox{.62}{ $#1$}}}\hspace{-2.2pt}}

\newcommand{\Anull}{\ps{0}A}
\newcommand{\Aplus}{\ps{\bm{+}}A}
\newcommand{\Aminus}{\ps{\bm{-}}A}
\newcommand{\Aeps}{^{\epsilon}\hspace{-3pt}A}
\newcommand{\Adelta}{^{\delta}\hspace{-3pt}A}
\newcommand{\Apm}{{}^{\scriptscriptstyle\pm}\hspace{-2pt}A}

\def\Eplus{\Erw_{\scriptscriptstyle +}}
\def\Eminus{\Erw_{\scriptscriptstyle -}}
\def\IHplus{\H_{\scriptscriptstyle +}}
\def\IHminus{\H_{\scriptscriptstyle -}}

\def\IPplus{\Prob_{\scriptscriptstyle +}}
\def\IPminus{\Prob_{\scriptscriptstyle -}}
\def\IUplus{\U_{\scriptscriptstyle +}}

\def\IUminus{\U_{\scriptscriptstyle -}}
\def\bUminus{\mathbf{U}_{\scriptscriptstyle -}}
\def\IUminplus{\U_{\scriptscriptstyle -}^{\scriptscriptstyle +}}

\def\supplus{^{\scriptscriptstyle (+)}}
\def\supminus{^{\scriptscriptstyle (-)}}
\def\suppm{^{\scriptscriptstyle (\pm)}}
\def\wtil{\widetilde}
\newcommand{\card}{\textsl{card}\hspace{.5pt}}

\def\ALIFS{\textsf{ALIFS}}
\def\AR{$\mathsf{AR}$}
\def\ARCH{$\mathsf{ARCH}$}
\def\IFS{$\mathsf{IFS}$}
\def\LIFS{$\mathsf{LIFS}$}
\def\MRW{$\mathsf{MRW}$}

\def\Lri{L^{\to}}
\def\Lle{L^{\leftarrow}}

\newcommand{\ppp}{p_{\scriptscriptstyle ++}}
\newcommand{\pppprime}{p_{\scriptscriptstyle ++}^{\,\prime}}
\newcommand{\pmm}{p_{\scriptscriptstyle --}}
\newcommand{\pmmprime}{p_{\scriptscriptstyle --}^{\,\prime}}
\newcommand{\ppm}{p_{\scriptscriptstyle +-}}
\newcommand{\pmp}{p_{\scriptscriptstyle -+}}
\newcommand{\pim}{\pi_{\scriptscriptstyle -}}
\newcommand{\pip}{\pi_{\scriptscriptstyle +}}
\newcommand{\whpim}{\wh{\pi}_{\scriptscriptstyle -}}
\newcommand{\whpip}{\wh{\pi}_{\scriptscriptstyle +}}

\newcommand{\deltam}{\delta_{\scriptscriptstyle -}}
\newcommand{\deltap}{\delta_{\scriptscriptstyle +}}
\def\uplus{u_{\scriptscriptstyle +}}
\def\uminus{u_{\scriptscriptstyle -}}
\def\vplus{v_{\scriptscriptstyle +}}
\def\vminus{v_{\scriptscriptstyle -}}
\def\wplus{w_{\scriptscriptstyle +}}
\def\wminus{w_{\scriptscriptstyle -}}
\def\kappaplus{\kappa_{\scriptscriptstyle +}}
\def\kappaminus{\kappa_{\scriptscriptstyle -}}

\def\Cplus{C_{\scriptscriptstyle +}}
\def\Cminus{C_{\scriptscriptstyle -}}
\def\Cminplus{C_{\scriptscriptstyle -}^{\scriptscriptstyle +}}

\allowdisplaybreaks[4]

\begin{document}

\title*{\Large Asymptotically linear iterated function systems on the real line}
\titlerunning{Asymptotically linear iterated function systems}

\author{Gerold Alsmeyer, Sara Brofferio and Dariusz Buraczewski}

\institute{{\sc Gerold Alsmeyer} \at Institute of Mathematical Stochastics, Department
of Mathematics and Computer Science, University of M\"unster,
Orl\'ean-Ring 10, D-48149 M\"unster, Germany.\at
\email{gerolda@math.uni-muenster.de}\\
{\sc Sara Brofferio} \at Universit\'e Paris-Saclay, CNRS, Laboratoire de Math\'ematiques d'Orsay, 91405 Orsay Cedex, et Universit\'e Paris-Est, CNRS, LAMA\\94010 Creteil, et Universit\'e Gustave Eiffel, LAMA\\77447 Marne-la-Vall\'ee, France\at
\email{sara.brofferio@math.u-pec.fr}\\
{\sc Dariusz Buraczewski} \at Institute of Mathematics, University of Wroclaw,
pl. Grunwaldzki 2/4, 50-384 Wroclaw, Poland.\at
\email{dbura@math.uni.wroc.pl}}

\maketitle

\abstract{
Given a sequence of i.i.d. random functions $\Psi_{n}:\R\to\R$, $n\in\N$, we consider the iterated function system and Markov chain which is recursively defined by $X_{0}^{x}:=x$ and $X_{n}^{x}:=\Psi_{n-1}(X_{n-1}^{x})$ for $x\in\R$ and $n\in\N$. Under the two basic assumptions that the $\Psi_{n}$ are a.s.~continuous at any point in $\R$ and asymptotically linear at the ``endpoints'' $\pm\infty$, we study the tail behavior of the stationary laws of such Markov chains by means of Markov renewal theory. Our approach provides an extension of Goldie's implicit renewal theory \cite{Goldie:91} and can also be viewed as an adaptation of Kesten's work on products of random matrices \cite{Kesten:73} to one-dimensional function systems as described. Our results have applications in quite different areas of applied probability like queuing theory, econometrics, mathematical finance and population dynamics, e.g.~\ARCH\  models and random logistic transforms.}

\bigskip



\section{Introduction}\label{sec:intro}

Let $\Psi,\Psi_{1},\Psi_{2},\ldots:\R\to\R$ be i.i.d. random functions, defined on a common probability space $(\Omega,\fA,\Prob)$,
such that $\Psi$ is a.s.~continuous at each $x\in\R$, i.e.
\begin{equation}\label{eq:Feller property}
\Prob[\omega:\Psi(\omega,\cdot)\text{ is continuous on } \R]\ =\ 1.
\end{equation}
 Then the associated iterated function system (\IFS), recursively defined by
\begin{equation}\label{eq:IFS}
X_{n}=\Psi_{n}(X_{n-1})\ =\ \Psi_{n}\cdots\Psi_{1}(X_{0})
\end{equation}
for $n\ge 1$, where $X_{0}$ is independent of the sequence $(\Psi_{n})_{n\ge 1}$ and $\Psi_{n}\cdots\Psi_{1}$ is used as shorthand for $\Psi_{n}\circ\ldots\circ\Psi_{1}$, forms a temporally homogeneous Markov chain on $\R$ which, by \eqref{eq:Feller property}, has the Feller property. For the case when $\Psi$ is asymptotically linear at $\pm\infty$ in the sense that
\begin{equation}\label{eq:def AL}
\sup_{x\le 0}|\Psi(x) - \Aminus x| \le B\quad\text{and}\quad
\sup_{x\ge 0}|\Psi(x) - \Aplus x| \le B
\end{equation}
for some real random variables $\Aplus,\,\Aminus, B$, the purpose of this article is to provide general conditions which
\begin{itemize}
\item ensure that $(X_{n})_{n\ge 0}$ possesses a stationary distribution $\nu$
\end{itemize}
and, a fortiori,
\begin{itemize}
\item allow to describe the tail behavior of $\nu$ at $\pm\infty$.
\end{itemize}
Instances of asymptotically linear \IFS, shortly called \ALIFS\ hereafter, appear in many contexts of applied probability and related fields like queueing models, econometrics, financial time series or population dynamics. The following known examples all fit into this class, at least after suitable conjugation $\Psi\rightsquigarrow g^{-1}\circ\Psi\circ g$ or extension of $\Psi$ from the positive halfline to the whole real line.
\begin{itemize}[leftmargin=.8cm]\itemsep1pt
\item[(i)] \emph{Random affine recursions}: $\Psi(x)=Ax+B$.
\item[(ii)] \emph{Lindley recursions}: $\Psi(x)=(Ax+B)^{+}$.
\item[(iii)]  \emph{ARCH(1) models}: $\Psi(x)=\left(\beta+\lambda x^{2}\right)^{1/2}Z$  with  $\beta,\lambda>0$.
\item[(iv)] \emph{\AR(1) models with \ARCH(1) errors}: $\Psi(x)=\alpha x+\left(\beta+\lambda x^{2}\right)^{1/2}Z$ with $\beta,\lambda>0$.
\item[(v)] \emph{Stochastic Beverton-Holt model}: $\Psi(x)=Ax/(1+x/B)$, $x>0$.
\item[(vi)] \emph{Random logistic transforms}: $\Psi(x)=A x(1-x)$, $x\in (0,1)$.
\end{itemize}
Here the greek letters are deterministic parameters whereas the capital letters $A,B,Z$ denote random variables, which in the last two examples are also supposed to be positive. In (vi), even $0<A<4$ must be assumed so as to guarantee that $\Psi$ forms a random self-map of $(0,1)$. Further examples of \ALIFS\ can be found in the survey papers by Aldous and Bandyopadhyay \cite{Aldous:B} and by Diaconis and Freedman \cite{Diaconis:Freedman}, and also in \cite[Section 6]{BroBura:13}.

To put our work into context, we first mention Kesten's \cite{Kesten:73} seminal paper on random affine recursions $X_{n}=A_{n}X_{n-1}+B_{n}$ on $\R^{d}$ (the multivariate version of (i) with i.i.d. $d\times d$ random matrices $A_{n}$ and $d$-dimensional random vectors $B_{n}$), where it is shown, under conditions ensuring positive recurrence, that the tail behavior of the unique stationary law $\nu$ of $(X_{n})_{n\ge 0}$ can be determined by use of renewal theory (after a change of measure) for an associated Markov random walk (\MRW). This walk is obtained upon approximating $X_{n}$ by a linear \IFS~$Z_{n}$ and then decomposing $Z_{n}$ into its distal part, given by the Euclidean norm $|Z_{n}|$, and its directional part $Z_{n}/|Z_{n}|$ which forms a recurrent Markov chain on the sphere $\bbS^{d-1}$. If $d=1$, the latter reduces to the finite set $\bbS^{0}=\{\pm 1\}$. A renewal-theoretic approach was also taken by Goldie \cite{Goldie:91} who studied the tail behavior of $\nu$ for one-dimensional, asymptotically linear $\Psi$ with $\Aplus=\Aminus$. We refer to a recent monograph \cite{BurDamMik:16} for an overview on random affine recursions.

\vspace{.1cm}
One of the central questions to be answered in the present work is about the impact of distinct $\Aplus,\Aminus$ on the left and right tail of $\nu$.
 This will be accomplished by employing Kesten's method in the one-dimensional setup where it applies without various tedious technicalities that occur in higher dimension. The reason for this simplification is that, as already mentioned, the directional part $Z_{n}/|Z_{n}|$ of $X_{n}$ takes values in $\{\pm 1\}$ only and thus reduces to a  simple finite Markov chain. More precisely, we will compare the given \ALIFS\ with an approximating \LIFS\ (for \emph{linear} IFS) of random linear functions and apply Kesten's method to the latter one. The comparison idea has already appeared in recent work by Mirek \cite{Mirek:11} and by the authors \cite{Alsmeyer:16,BroBura:13}. Our approach may also be viewed as an extension of Goldie's implicit renewal theory, the extension being that the random walk in Goldie's approach is now Markov-modulated and thus a \MRW. We will return to this point with more explanations later.

\section{Basic assumptions and main results}\label{sec:main results}
Our standing assumption \eqref{eq:def AL} on $\Psi$ throughout this work can be expressed in the more compact form
\begin{gather}
\sup_{x\in\R}|\Psi(x)-\Lambda(x)|\ \le\ B\quad\text{a.s.}\label{eq:def2 AL}
\shortintertext{where}
\label{eq:def Lambda(x)}
\Lambda(x)\,:=\,{}^{\sign(x)}\!A x
\ =\ \begin{cases} \Aplus x,&\text{if }x>0,\\ \Aminus x,&\text{if }x<0,\\ \hfill 0,&\text{if }x=0.
\end{cases}
\end{gather}
for some real-valued random variables $\Aminus,\Aplus$ and $B$ such that, without loss of generality, $B\ge 1$. We further put $\Anull:=0$ and $\sign(x):=\1_{\IRg}(x)-\1_{\IRl}(x)$ for $x\in\R$, where  $\IRl:=(-\infty,0)$ and $\IRg:=(0,\infty)$.
In other words, we are given a sequence
$$ (\Psi_{n},\Lambda_{n},\!\Aminus_{n},\!\Aplus_{n},B_{n}),\quad n=1,2,,\ldots $$
of i.i.d. copies of $(\Psi,\Lambda,\Aminus,\Aplus,B)$ satisfying \eqref{eq:def Lambda(x)} and consider the Markov chain defined by
\eqref{eq:IFS}.

\begin{figure}[t]
\includegraphics[width=0.38\textwidth]{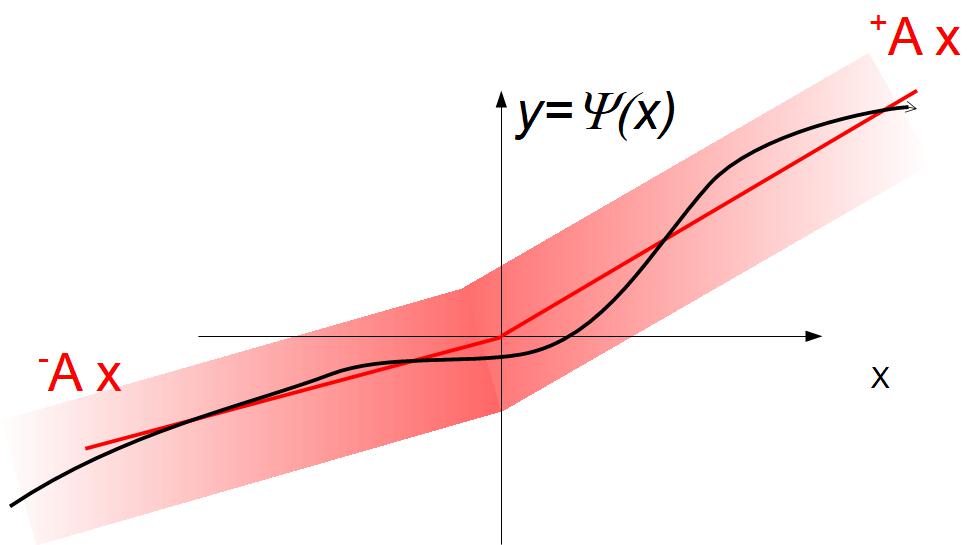}\hspace{1cm}\includegraphics[width=0.38\textwidth]{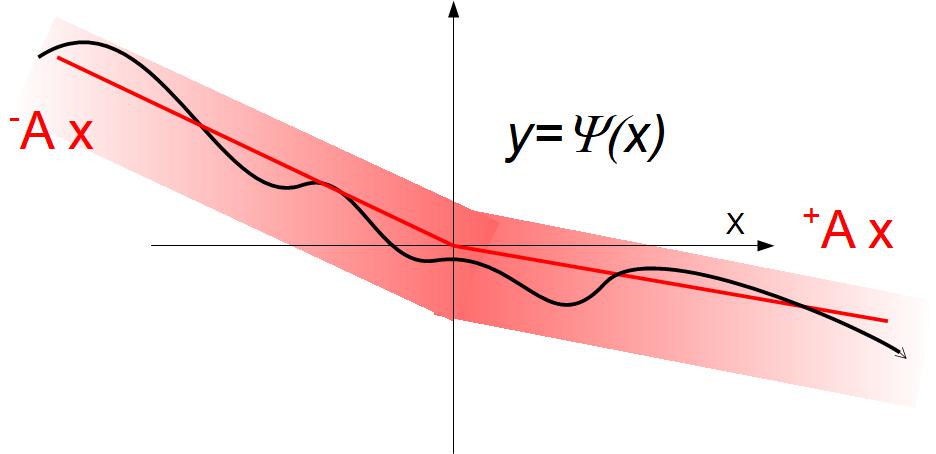}\\
\vspace{.3cm}
\includegraphics[width=0.38\textwidth]{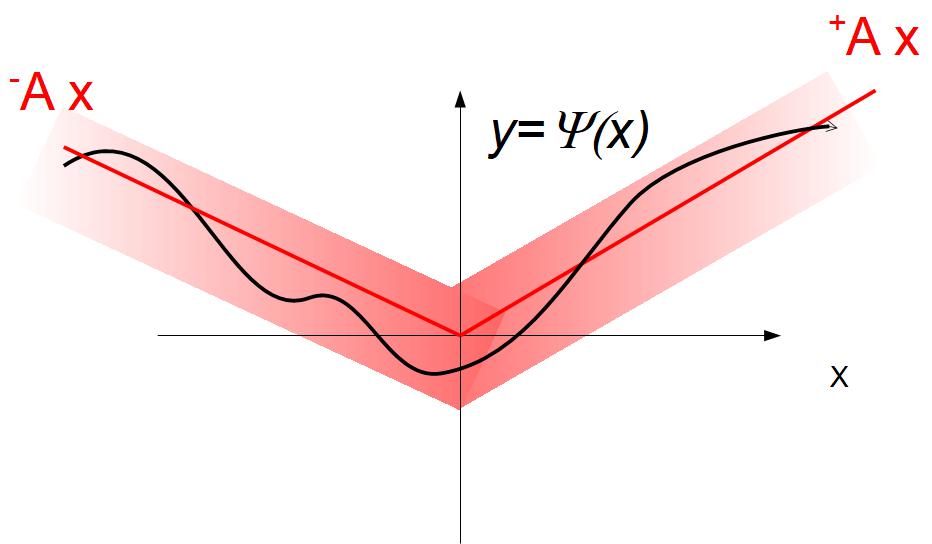}\hspace{1cm}\includegraphics[width=0.38\textwidth]{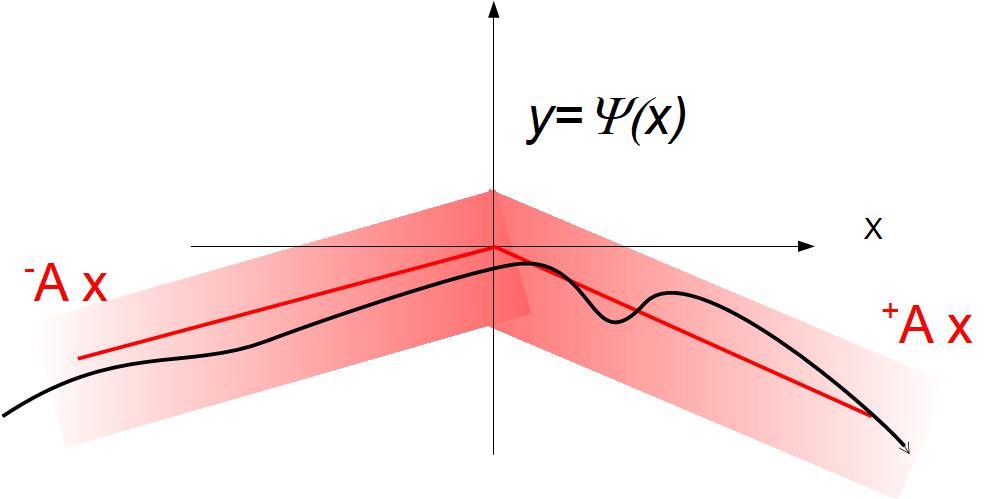}
\caption{The four possible shapes of $\Psi$.\newline
\textbf{Slash-type \textbf{$/\hspace{-4.27pt}/$}} (top left): ${}^{\scriptstyle -}\!A>0$ and ${}^{\scriptstyle +}\!A>0$, thus $-\Psi(-\infty)=\Psi(+\infty)=+\infty$.\newline\noindent
\textbf{Backslash-type \textbackslash} (top right): ${}^{\scriptstyle -}\!A<0$ and ${}^{\scriptstyle +}\!A<0$, thus $\Psi(-\infty)=-\Psi(+\infty)=+\infty$.\newline\noindent
\textbf{Vee-type $\bigvee$} (bottom left): ${}^{\scriptstyle -}\!A<0<{}^{\scriptstyle +}\!A$, thus $\Psi(-\infty)=\Psi(+\infty)=+\infty$.\newline\noindent
\textbf{Wedge-type $\bigwedge$} (bottom right): ${}^{\scriptstyle +}\!A<0<{}^{\scriptstyle -}\!A$, thus $\Psi(-\infty)=\Psi(+\infty)=-\infty$.}\label{fig1}
\end{figure}

\vspace{.1cm}
Provided that the observed values of $\Aminus,\Aplus$ are both nonzero, the pertinent realization of $\Psi$ as a function may, regarding its overall shape, exhibit one of four distinct types as illustrated in Fig.~\ref{fig1}. These are denoted mnemonically as slash-type $/$, backslash-type $\setminus$, vee-type $\vee$, and wedge-type $\wedge$. In the simple affine model with $\Aplus=\Aminus=A>0$, the function $\Psi$ is always of slash-type. Goldie's implicit renewal theory \cite{Goldie:91} is designed for \ALIFS\ with $\Psi$ satisfying
$$ |\Psi(x)-Ax|\ \le\ B $$
for some random variables $A,B$. It therefore mixes functions of slash-type and backslash-type such that $\Aminus=\Aplus$. The \AR(1)-model with \ARCH(1) errors provides an instance where functions of type $/$, $\bigvee$ and $\bigwedge$ are mixed. If the function $\Psi(x)$ is uniformly bounded for $x>0$ (resp. $x<0$), then  $\Aplus=0$ (resp. $\Aminus=0$). This occurs, for instance, in the Beverton-Holt model.

\vspace{.2cm}
In view of \eqref{eq:def2 AL}, it is natural to relate the \ALIFS\ $X_{n}=\Psi_{n}\cdots\Psi_{1}(X_{0})$ with the \LIFS\ $\Lambda_{n}\cdots\Lambda_{1}(X_{0})$ which in turn, following Kesten's approach in the present one-dimensional setting, can be studied with the help of a suitable temporally homogeneous Markov chain $\Xi=(\xi_{n})_{n\ge 0}$. Namely, let $\xi_{0}:=\sign(X_{0})$ and
\begin{equation}\label{eq:def xi_{n}}
\xi_{n}\,:=\,\sign(\Lambda_{n}(\xi_{n-1}))\ =\
\begin{cases}
\sign(\Aminus_{n})\xi_{n-1},&\text{if }\xi_{n-1}=-1,\\
\sign(\Aplus_{n})\xi_{n-1},&\text{if }\xi_{n-1}=1,\\
\hfill 0,&\text{if }\xi_{n-1}=0,
\end{cases}
\end{equation}
for $n\ge 1$. This chain has state space $\cS=\bbS^{0}\cup\{0\}$, and the state 0, if it appears, is absorbing in which case at least one of the states $\pm 1$ must be transient. Whenever convenient, the set $\bbS^{0}$ is identified with the set of signs $\{-,+\}$ (e.g., in sub- or superscripts as in \eqref{eq:def Lambda(x)}) because $\xi_{n}$ keeps track of the sign of $\Lambda_{n}\cdots\Lambda_{1}(X_{0})$, that is $\xi_{n}=\sign(\Lambda_{n}\cdots\Lambda_{1}(X_{0}))$ with $\Lambda_{n}\cdots\Lambda_{1}$ used as shorthand for $\Lambda_{n}\circ\cdots\circ\Lambda_{1}$. Let
$$ p_{\delta\epsilon}\,:=\,\Prob[\xi_{n}=\epsilon|\xi_{n-1}=\delta]\ =\ \Prob[\sign(\Adelta)\delta=\epsilon] $$
for $\delta,\epsilon\in\{-1,0,+1\}$, whence the possibly reduced and therefore substochastic transition matrix of $\Xi$ on $\bbS^{0}$ is given by
\begin{align}\label{eq:def of P}
P\ =\ \begin{pmatrix} \pmm &\pmp\\ \ppm &\ppp \end{pmatrix}
\ =\ \begin{pmatrix}
\Prob[\Aminus>0] &\Prob[\Aminus<0]\\ \Prob[\Aplus<0] &\Prob[\Aplus>0]
\end{pmatrix}.
\end{align}
As common, we put $\Prob_{\delta}:=\Prob[\cdot|\xi_{0}=\delta]$ and $\Prob_{\chi}:=\sum_{\delta\in\cS}\chi_{\delta}\Prob_{\delta}$ for any measure $\chi$ on $\cS$.

\vspace{.2cm}
In order to state our main results on the tail behavior of any stationary distribution of the given \ALIFS\ $(X_{n})_{n\ge 0}$, we distinguish three cases regarding the transition structure of the chain $\Xi$ (see Fig. \ref{fig11}).
\begin{description}
\item[\bf Case 1 (irreducible case)]  $p_{-+}>0$ and $p_{+-}>0$, that is both $\Aplus$ and $\Aminus$ are negative with positive probability. We will see that the tails of an invariant distribution $\nu$ at $+ \infty$ and $-\infty$ are of the same order in this case.
\item[\bf Case 2 (unilateral case)] $p_{-+}>0$ and $p_{+-}=0$ , that is $\Aminus$ is negative with positive probability but $\Aplus\ge 0$. The functions $\Psi$ are only of types $/$ and $\vee$. In this case, the order of decay of $\nu$ at $+ \infty$ can depend on both coefficients $\Aplus$ and $\Aminus$, while  the behavior at $-\infty$ depends only on $\Aminus$. The corresponding case  $p_{+-}>0$ and $p_{-+}=0$ can be treated without further ado after conjugation by $x\mapsto-x$.
\item[\bf Case 3 (separated case)] $p_{-+}=0$ and $p_{+-}=0$ , that is $\Aminus\ge 0$ and $\Aplus\ge 0$. The functions $\Psi$ are only of type $/$. In this case, the order of decay of the tail of $\nu$ at $+\infty$ (resp. $-\infty$) depends only on $\Aplus$ (resp. $\Aminus$).
\end{description}

\begin{figure}
 \begin{tikzpicture}[font=\sffamily]
        \tikzset{node style/.style={state,
                                    minimum width=.4cm, 
                                    line width=.2mm,
                                    fill=yellow!70!white}}

        \node[node style] at (0, 0)     (minus)     {\large --};
        \node[node style] at (2.5, 0)     (plus)     {\large +};
        \node[node style] at (1.25, -2) (zero) {\large 0};

        \draw[every loop,
              auto=right,
              line width=.4mm,
              >=latex,
              draw=blue,
              fill=blue]
            (minus)     edge[bend right=0,dashed]            node{} (zero)
            (plus)     edge[bend right=0,dashed]            node{}  (zero)
            (plus)     edge[bend left]            node {} (minus)
            (minus)     edge[bend left]            node {} (plus);
        \node[node style] at (4, 0)     (minus)     {\large --};
        \node[node style] at (6.5, 0)     (plus)     {\large +};
        \node[node style] at (5.25, -2) (zero) {\large 0};

        \draw[every loop,
              auto=right,
              line width=.4mm,
              >=latex,
              draw=blue,
              fill=blue]
            (minus)     edge[bend right=0,dashed]            node{} (zero)
            (plus)     edge[bend right=0,dashed]            node{}  (zero)
            (minus)     edge[bend left]            node {} (plus);
        \node[node style] at (8, 0)     (minus)     {\large --};
        \node[node style] at (10.5, 0)     (plus)     {\large +};
        \node[node style] at (9.25, -2) (zero) {\large 0};

        \draw[every loop,
              auto=right,
              line width=.4mm,
              >=latex,
              draw=blue,
              fill=blue]
            (minus)     edge[bend right=0,dashed]            node{} (zero)
            (plus)     edge[bend right=0,dashed]            node{}  (zero);
    \end{tikzpicture}
    \caption{The transition structure of $\Xi$ in the three cases 1--3. The dashed arrows indicate transitions that may have both positive or zero probability.}
    \label{fig11}
\end{figure}
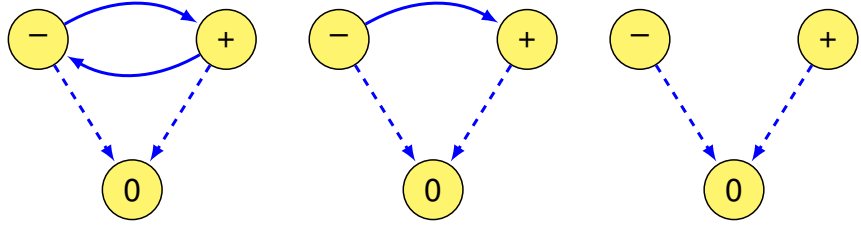

Fundamental tools in our study are
the Cram\'er transform of $P$, defined by
\begin{equation}\label{eq:def P(theta)}
P(\theta)\,:=\,\begin{pmatrix} \pmm(\theta) &\pmp(\theta)\\ \ppm(\theta) &\ppp(\theta) \end{pmatrix}
\,:=\,\begin{pmatrix}
\Erw|\Aminus|^{\theta}\1_{\{\Aminus>0\}} &\Erw|\Aminus|^{\theta}\1_{\{\Aminus<0\}}\\ \Erw|\Aplus|^{\theta}\1_{\{\Aplus<0\}} &\Erw|\Aplus|^{\theta}\1_{\{\Aplus>0\}}
\end{pmatrix}
\end{equation}
for $\theta\in\bbD:=\{\vth\ge 0:\Erw|\Aminus|^{\vth}+\Erw|\Aplus|^{\vth}<\infty\}$, and its dominant eigenvalue (spectral radius) $\rho(\theta)$. They will be discussed in greater detail in Subsection \ref{subsec:Cramer transform}

\vspace{.2cm}
\textit{\bfseries Case 1 (irreducible case)}. $p_{-+}>0$ and $p_{+-}>0$.\\[1mm]
In this case, the dominant eigenvalue $\rho(\theta)$ is associated to the left and right nonnegative eigenvectors
$$ u(\theta)\,=\,(\uminus(\theta),\uplus(\theta))^{\top}\quad\text{and}\quad v(\theta)\,=\,(\vminus(\theta),\vplus(\theta))^{\top}, $$
respectively, uniquely determined by $u(\theta)^{\top}v(\theta)=1$ and $\uplus(\theta)+\uminus(\theta)=1$. This implies that
\begin{equation}\label{eq:def of pihat}
\wh{\pi}(\theta)\,:=\,\big(\uminus(\theta)\vminus(\theta),\uplus(\theta)\vplus(\theta)\big)^{\top}.
\end{equation}
forms a probability distribution. For $\theta>0$, it will later be identified as the stationary law of $(\xi_{n})_{n\ge 0}$ after a suitable change of measure. For $\theta=0$, this is only true when $P$ is stochastic or, equivalently,
\begin{equation}\label{eq:Aeps both not zero}
\Prob[\Aminus=0]\ =\ \Prob[\Aplus=0]\ =\ 0
\end{equation}
holds in which case
\begin{equation}\label{eq:def of pi}
\pi\ =\ \left(\frac{\ppm}{\pmp+\ppm},\frac{\pmp}{\pmp+\ppm}\right)^{\top}\ =\ \wh{\pi}(0)
\end{equation}
equals the unique associated stationary distribution of $\Xi$.

\vspace{.1cm}
The crucial assumption in the subsequent theorem is that
\begin{equation}\label{eq:def of kappa}
\rho(\kappa)\,=\,1\quad\text{for some }\kappa>0.
\end{equation}
We denote by $\cC^{*}(\R)$ the space of bounded Lipschitz functions $\phi$ on $\R$ which vanish in a neighborhood of the origin and by $\cC_{-}^{*}(\R),\,\cC_{+}^{*}(\R)$ the subspaces of those $\phi$ that also vanish on $\IRge,\,\IRle$, respectively.

\begin{Theorem}\label{thm:main case 1}
Assuming \eqref{eq:def of kappa},
\begin{gather}
\Erw|{}^{\delta}\!A|^{\kappa}\log|{}^{\delta}\!A|\,<\,\infty\text{ for $\delta \in \{-,+\}$,}\quad\Erw B^{\kappa}\,<\,\infty\label{eq:moments AlogA B}
\shortintertext{and}
\Prob_{\wh{\pi}(\kappa)}[\log|{}^{\xi_{0}}\!A|-a_{\xi_{1}}+a_{\xi_{0}}\in d\Z]\,<\,1\quad\text{for all }d>0\text{ and }a_{\pm}\in\R,\label{eq:log A nonlattice}
\end{gather}
any stationary distribution $\nu$ of the \textsf{ALIFS} $(X_{n})_{n\ge 0}$ has power tails of order $\kappa$, more precisely
\begin{equation}\label{eq:tails main case 1}
\lim_{t\to\infty}t^{\kappa}\,\nu((t,\infty))\,=\,\Cplus\quad\text{and}\quad\lim_{t\to\infty}t^{\kappa}\,\nu((-\infty,-t))\,=\,\Cminus
\end{equation}
for constants $\Cplus,\Cminus \ge 0$ which are explicitly defined in \eqref{eq:def Cplus case 1},\eqref{eq:def Cminus case 1} and satisfy
\begin{equation}\label{eq:relation Cplusminus}
\uplus(\kappa)\,\Cminus\ =\ \uminus(\kappa)\,\Cplus.
\end{equation}
Furthermore,
\begin{equation}\label{eq:KRT main case 1}
\lim_{t\to\infty}t^{\kappa}\int\phi(t^{-1}x)\ \nu(dx)\ =\ \kappa\int_{0}^{\infty}\frac{\Cminus\phi(-x)+\Cplus\phi(x)}{x^{\kappa+1}}\ dx
\end{equation}
for any $\phi\in\cC^{*}(\R)$.
\end{Theorem}

Further information on the lattice-type Condition \eqref{eq:log A nonlattice} will be provided later, see Subsect.~\ref{subsec:MRT}.

Observe that the above result concerns the asymptotic behavior of the tail of the stationary measure. Our proof, which is based on a renewal theorem, does not entail the positivity of the limiting constants $\Cplus,\Cminus$. However, we will be able to verify this in Section \ref{sec:positivity} under the additional assumption that the stationary measure has unbounded support (see Proposition \ref{prop:tail constants positive}).

\vspace{.2cm}
\textit{\bfseries Case 2 (unilateral case)}. $p_{-+}>0$ and $p_{+-}=0$.\\[1mm]
In this case, the crucial numbers, if they exist, are $\kappaminus,\kappaplus>0$ defined by
\begin{equation*}
\pmm(\kappa_{-})\,=\,1\quad\text{and}\quad\ppp(\kappaplus)\,=\,1.
\end{equation*}
As a substitute for Condition \eqref{eq:moments AlogA B}, we need that either
\begin{gather}
\Erw|\Aminus|^{\kappa}\log|\Aminus|\,<\,\infty\quad\text{and}\quad\Erw B^{\kappa}\,<\,\infty\label{eq:moments AlogAminus B}
\shortintertext{or}
\Erw|\Aplus|^{\kappa}\log|\Aplus|\,<\,\infty\quad\text{and}\quad\Erw B^{\kappa}\,<\,\infty.\label{eq:moments AlogAplus B}
\end{gather}

\begin{Theorem}\label{thm:main case 2}
(a) If $\kappaminus$ exists, Condition \eqref{eq:moments AlogAminus B} holds for $\kappa=\kappaminus$ and $\P[\log|\Aminus|\in\cdot|\Aminus>0]$ is nonarithmetic, then any stationary distribution $\nu$ satisfies
\begin{gather}
\lim_{t\to\infty}t^{\kappaminus}\,\nu((-\infty,-t))\,=\,\Cminus\label{eq:left tails main case 2}
\shortintertext{as well as}
\lim_{t\to\infty}t^{\kappaminus}\int\phi(t^{-1}x)\ \nu(dx)\ =\ \kappaminus\int_{0}^{\infty}\frac{\Cminus\phi(-x)}{x^{\kappaminus+1}}\ dx\label{eq:KRT1 main case 2}
\end{gather}
for $\phi\in\cC_{-}^{*}(\R)$, where $\Cminus$ is defined in \eqref{Cminus Thm 2.2(a)}.

\vspace{.1cm}
(b) If $\kappaplus$ exists, $\pmm(\kappaplus)<1$ (thus $\kappaminus$, if it exists, is greater than $\kappaplus$), $\pmp(\kappaplus)<\infty$, Condition \eqref{eq:moments AlogAplus B} holds for $\kappa=\kappaplus$ and $\P[\log|\Aplus|\in\cdot|\Aplus>0]$ is nonarithmetic, then any stationary distribution $\nu$ satisfies
\begin{gather}
\lim_{t\to\infty}t^{\kappaplus}\,\nu((t,\infty))\,=\,\Cplus\label{eq:right tails main case 2A}
\shortintertext{as well as}
\lim_{t\to\infty}t^{\kappaplus}\int\phi(t^{-1}x)\ \nu(dx)\ =\ \kappaplus\int_{0}^{\infty}\frac{\Cplus\phi(x)}{x^{\kappaplus+1}}\ dx\label{eq:KRT2 main case 2}
\end{gather}
for $\phi\in\cC^{*}(\R)$, where $\Cplus$ is defined in \eqref{Cplus Thm 2.2(b)}.

\vspace{.1cm}
(c) If $\kappaminus$ exists, $\ppp(\kappaminus)<1$ (thus $\kappaminus<\kappaplus$ if the latter exists as well), $\ppp(\theta)<1$ and $\pmp(\theta)<\infty$ for some $\theta> \kappaminus$, Condition \eqref{eq:moments AlogAminus B} holds for $\kappa=\kappaminus$ and $\IPminus[\log|\Aminus|\in\cdot|\Aminus>0]$ is nonarithmetic, then any stationary distribution $\nu$ satisfies
\begin{gather}
\lim_{t\to\infty}t^{\kappaminus}\,\nu((t,\infty))\ =\ \Cminplus\label{eq:right tails main case 2B}
\shortintertext{as well as}
\lim_{t\to\infty}t^{\kappaminus}\int\phi(t^{-1}x)\ \nu(dx)\ =\ \kappaminus\int_{0}^{\infty}\frac{\Cminus\phi(-x)+\Cminplus\phi(x)}{x^{\kappaminus+1}}\ dx\label{eq:KRT3 main case 2}
\end{gather}
for $\phi\in\cC^{*}(\R)$, where $\Cminus,\,\Cminplus$ are defined in \eqref{Cminus Thm 2.2(a)} and \eqref{Cminplus Thm 2.2(c)}, respectively.
\end{Theorem}

Observe that if both $\kappaminus$ and $\kappaplus$ exist with $\kappaminus > \kappaplus$, then cases (a) and (b) entail that the stationary measure behaves regularly at $+\infty$ and $-\infty$, but with different tail decay rates.

\vspace{.1cm}
The unilateral case shares several features with the study of the stationary distribution of the two-dimensional recursive Markov chain defined by the affine recursions $\Psi_{n}(x)=A_{n}x+B_{n}$ in the case when the $A_{n}$ are upper triangular matrices. Such models have attracted some interest in recent years due to their relevance for some models in econometrics, see
\cite{DMS19}.

\vspace{.1cm}
Regarding the case $\kappa_{+}=\kappa_{-}$, we further remark that the methods used in the present paper are not strong enough and thus need to be refined. It seems reasonable to believe that in this case a first order expansion of $P(\theta)$ in $\kappa$ is required with a possible extra term in the power tail of $\nu$, see again \cite{DZ18} for similar considerations in the case of the afore-mentioned two-dimensional recursions.

\vspace{.2cm}
\textit{\bfseries Case 3 (separated case)}. $\pmp=0$ and $\ppm=0$.\\[1mm]
This is the easiest case and can be treated by Goldie's implicit renewal theory \cite{Goldie:91}. We state the result here for completeness and put
$$ C_{\delta}\,:=\,\frac{\Erw\big[|\Psi(R)|^{\kappa_{\delta}}\1_{\{\Psi(R)<0\}}-|\Adelta R|^{\kappa_{\delta}}\1_{\{\delta R>0\}}\big]}{\kappa_{\delta}\,p_{\delta\delta}'(\kappa_{\delta})} $$
for $\delta\in\{-,+\}$.

\begin{Theorem}\label{thm:main case 3}
(a) If $\kappaminus$ exists, Condition \eqref{eq:moments AlogAminus B} holds for $\kappa=\kappaminus$ and $\log|\Aminus|$ is nonarithmetic, then any stationary distribution $\nu$ satisfies
\begin{align}\label{eq1:main case 3}
\lim_{t\to\infty}t^{\kappaminus}\,\nu((-\infty,-t))\,=\,\Cminus
\end{align}
where $R$ has law $\nu$ and is independent of $\Psi,\Aminus$. Moreover, \eqref{eq:KRT1 main case 2} holds for any $\phi\in\cC_{-}^{*}(\R)$ and with $\Cminus$ as in \eqref{eq1:main case 3}.

\vspace{.1cm}
(b) If $\kappaplus$ exists, Condition \eqref{eq:moments AlogAplus B} holds for $\kappa=\kappaplus$ and $\log|\Aplus|$ is nonarithmetic, then any stationary distribution $\nu$ on $\IRg$ satisfies
\begin{align}\label{eq2:main case 3}
\lim_{t\to\infty}t^{\kappaplus}\,\nu((t,\infty))\,=\,\Cplus
\end{align}
here $R$ has law $\nu$ and is independent of $\Psi,\Aplus$. Moreover, \eqref{eq:KRT2 main case 2} holds for any $\phi\in\cC_{+}^{*}(\R)$ and with $\Cplus$ as in \eqref{eq2:main case 3}.
\end{Theorem}

In any of the three cases the conditions \eqref{eq:def of kappa}  and \eqref{eq:moments AlogA B} ensure the existence of at least one stationary distribution of $(X_{n})_{n\ge 0}$. We refer to Section \ref{sec:existence} for details.

\subsection{Examples}\label{subsec:examples}

Let us briefly discuss some examples of \ALIFS\ that have appeared in the literature and whose stationary distributions exhibit a tail behavior that, under appropriate conditions, can be read off from our main results. In order to keep this presentation short, we refrain from giving any technical details. Further applications with a more thorough discussion can be found in
\cite{BroBura:13}.

\subsubsection{\ARCH(1)}  Our first example, the \emph{autoregressive conditional heteroskedasticity model of order one}, is well-known in econometrics and usually defined by the pair of recursive equations
$$ X_{n}\,=\,\Sigma_{n}Z_{n}\quad\text{and}\quad\Sigma_{n}^{2}\,=\,\beta+\lambda X^{2}_{n-1}, $$
where $(Z_{n})_{n\ge 1}$ denotes a sequence of i.i.d.~random variables with mean zero and variance one (the noise) and $\beta,\lambda$ are positive parameters. Simple inspection shows that this entails the recursive relation $X_{n}=\Psi_{n}(X_{n-1})$ for $n\ge 1$ and
i.i.d.~copies $\Psi_{1},\Psi_{2},\ldots$ of the random function
$$ \Psi(x)\,=\,Z\sqrt{\beta+\lambda x^{2}}. $$
As one can also readily check, $(X_{n})_{n\ge 0}$ forms an \ALIFS~which satisfies \eqref{eq:def2 AL} with $\Apm_{n}=\pm Z_{n}\sqrt{\lambda_{n}}$ and is irreducible (Case 1). Unless $Z$ has a symmetric law, the constants $C_{+}, C_{-}$ defined in \eqref{eq:tails main case 1} are generally distinct. Let us also point out here that $(X_{n})_{n\ge 0}$ remains an \ALIFS~if the parameters $\beta_{n}$ and $\lambda_{n}$ are allowed to be random.

\subsubsection{\AR(1) models with \ARCH(1) errors} This extension of the previous example is obtained by adding an extra linear term and therefore defined as the \ALIFS\ generated by i.i.d. copies of the random function
$$ \Psi(x)\,=\,\alpha x+Z\left(\beta+\lambda x^{2}\right)^{1/2} $$
for some $(\alpha,\beta,\lambda)\in\R\times\IRg^{2}$ and a random variable $Z$ as before. It satisfies \eqref{eq:def2 AL} with $\Apm=\alpha\pm Z\sqrt{\lambda}$. Depending on the parameters $\alpha,\lambda$ and the almost sure range of $Z$, all three cases introduced above may occur. We will return to this example in Section \ref{sec:AR(1) with ARCH errors} at the end of this work.

\subsubsection{\IFS\ on the unit interval} Consider an \IFS~generated by i.i.d.~copies of a random continuous self-map $\Phi$ of the unit interval $[0,1]$ which further satisfies $\Phi((0,1))\subseteq (0,1)$. If $\Phi$ is twice differentiable at $0$ and $1$, this \IFS\ can be conjugated to obtain an \ALIFS\ of the real line. Namely, by taking the diffeomorphism $r$ of $(0,1)$ onto $\R$, defined by
$$ r(u)\,:= \,-\frac{1}{u} + \frac{1}{1-u}, $$
the conjugated function $\Psi =  r\circ \Phi \circ r^{-1}  $ satisfies \eqref{eq:def2 AL} with
$$ \Aminus\,=\,\begin{cases}
\frac{1}{\Phi'(0)}&\text{if }\Phi(0)=0\text{ or }1\\
\hfill 0&\text{if } \Phi(0)\in(0,1)
\end{cases}\quad\text{and}\quad
      \Aplus\,=\,\begin{cases}
\frac{1}{\Phi'(1)}&\text{if } \Phi(1)=0\text{ or }1\\
\hfill 0&\text{if } \Phi(1)\in(0,1)
\end{cases}, $$
(see Section 6.3 in \cite{BroBura:13} for more details).
Note further that $\nu$ is an invariant distribution for the \IFS~generated by $\Phi$ (i.e.~$\Phi(X)\eqdist\nu$ if $X\eqdist\nu$, where $\eqdist$ means equality in law) iff $r*\nu$ is invariant for the \ALIFS~generated by $\Psi$.
Thus, under appropriate hypotheses, this system possesses a stationary distribution whose behaviour close to the boundaries of the interval can be deduced from our main results.
As a particular instance which has received some attention in the literature, we mention here the random logistic transform $\Phi(x)=Ax(1-x)$ with $0<A<4$ a.s., $\Aminus=1/\alpha$ and $\Aplus=1/\alpha$, see e.g. \cite{AthreyaDai:00}.

\subsection{Structure of the paper}

The principal goal of this work is to describe the tail behavior of a stationary measure $\nu$ of an \ALIFS\ at $\pm\infty$. In Sections \ref{sec:prerequisites} and \ref{sec:transfer operators}, we provide the indispensable tools to prove our main results. Then we prove the existence of the limits in Theorems \ref{thm:main case 1}, \ref{thm:main case 2} and \ref{thm:main case 3} in Sections \ref{sec:proof main case 1}, \ref{sec:proof main case 2} and \ref{sec:proof main case 3}, respectively. It will be seen that the nondegeneracy of the limits, i.e., the positivity of the limiting constants requires different arguments and in fact forms a separate problem. This question is postponed until Section \ref{sec:positivity} and there taken care of in Proposition \ref{prop:tail constants positive}.  In Section \ref{sec:existence}, we give conditions for the existence of at least one stationary distribution $\nu$ (see Prop.~\ref{prop:stationary distribution}) which are directly seen to be valid in our main theorems. Uniqueness of $\nu$ will not be discussed here, because this requires geometric arguments and a local analysis of the process  which in such generality is beyond the scope of this work. The AR(1) model with GARCH errors is an example which has received some interest in the literature \cite{Alsmeyer:16,BorkovecKl:01,GueDie:94,Maercker:97} and to which our results can be applied, in fact with all three cases being possible. It will therefore be discussed in greater detail in the final Section \ref{sec:AR(1) with ARCH errors}, followed by a short appendix containing a technical lemma about the maximal eigenvalue $\rho(\theta)$ in a right neighborhood of 0.


\section{Prerequisites}\label{sec:prerequisites}

\subsection{The induced Markov random walk}\label{subsec:main results}

Defining ${S_{0}}:=\log|X_{0}|$ and
\begin{equation*}
 S_{n}\,:=\,\log |\Lambda_{n}\cdots\Lambda_{1}(X_{0})|
\end{equation*}
for $n\ge 1$ (with the usual convention $\log 0:=-\infty$), we see that, given $\Xi$, the increments $\zeta_{n}:=S_{n}-S_{n-1}=\log|\Lambda_{n}(\xi_{n-1})|$, $n=1,2,,\ldots$, are conditionally independent and
\begin{align*}
\Prob[\zeta_{n}\in\cdot|\Xi,\xi_{n-1}=\delta,\xi_{n}=\epsilon]\ &=\ \Prob[\log|\Lambda_{n}(\xi_{n-1})|\in\cdot|\xi_{n-1}=\delta,\xi_{n}=\epsilon]\\
&=\ \Prob[\log|\Adelta|\in\cdot|\delta\cdot\sign(\Adelta)=\epsilon].
\end{align*}
Equivalently, $(\xi_{n},\zeta_{n})_{n\ge 0}$ forms a Markov chain such that the conditional law of $(\xi_{n},\zeta_{n})$ given the past depends on $\xi_{n-1}$ only. The transition kernel equals
\begin{equation*}
Q(\delta,\{\epsilon\}\times B)\ =\ \Prob[\delta\cdot\sign(\Adelta)= \epsilon,\,\log|\Adelta|\in B]
\end{equation*}
for measurable $B\subset\R$, $\delta\in\bbS^{0}$ and $\epsilon \in \cS$.  
If $\delta=0$, then $Q(\delta,\cdot)$ equals Dirac measure at $\dagger:=(0,-\infty)$. In other words, $\dagger$ is an absorbing state for $(\xi_{n},\zeta_{n})_{n\ge 0}$ and should be viewed as a grave. It follows that $(\xi_{n},S_{n})_{n\ge 0}$ does indeed constitute a Markov random walk (\MRW) with discrete driving chain $\Xi$ and induced by $(\Lambda_{n}\cdots\Lambda_{1}(X_{0}))_{n\ge 1}$. However, it may be absorbed at $\dagger$ in finite time (explosion of the additive part). On the other hand, the conditions in our main Theorems \ref{thm:main case 1}--\ref{thm:main case 3} ensure that, after a suitable change of measure $\Prob\rightsquigarrow\widehat{\Prob}$ to be decribed in Subsection \ref{subsec:measure change}, the driving chain has state space $\bbS^{0}$ and explosion does no longer occur. The relevant renewal-theoretic properties of the \MRW\ after this measure change, which is essential for the analysis of the tails of the stationary distributions of $(X_{n})_{n\ge 0}$, will be discussed in Subsection \ref{subsec:MRT}.

\subsection{Standard model}\label{subsec:2.2}

It is convenient to assume a standard model
$$ (\Omega,\fA,(\Prob_{x})_{x\in\R},(\xi_{n},S_{n})_{n\ge 0}), $$
where $(\Omega,\fA)$ denotes the measurable space on which all occurring random variables are defined and $\Prob_{x}:=\Prob[\cdot|X_{0}=x]$, thus
$$ \Prob_{x}[X_{0}=x,\xi_{0}=\sign(x),S_{0}=\log|x|]\ =\ 1 $$
for all $x\in\R$. The definition extends the one given before in Thm.~\ref{thm:main case 1} in a compatible way because $\Prob_{\delta}[\xi_{0}=\delta]=1$ if $\delta\in\cS$. Moreover, we put $\Prob_{\chi}:=\int\Prob_{x}\,\chi(dx)$ for any measure $\chi$ on $\R$ and use $\Prob$ for probabilities that do not depend on initial conditions.

\subsection{Associated \LIFS}\label{sec:associated LIFS}

The fact that $\Psi_{1},\Psi_{2},\ldots$ are i.i.d. implies that, for each $n\in\N_{0}$, the forward iteration $X_{n}$ and the backward iteration $\wh{X}_{n}:=\Psi_{1}\cdots\Psi_{n}(X_{0})$ have the same distribution, more precisely
\begin{equation}\label{eq:X_{n} eqdist hatX_{n}}
X_{n}\ =\ \Psi_{n}\cdots\Psi_{1}(X_{0})\ \eqdist\ \Psi_{1}\cdots\Psi_{n}(X_{0})\ =\ \wh{X}_{n}\quad\text{under }\Prob_{x}
\end{equation}
for each $n\in\N_{0}$. By a similar argument and in analogy with \eqref{eq:X_{n} eqdist hatX_{n}},
\begin{equation*}
\Lambda_{n}\cdots\Lambda_{1}(x)\ \eqdist\ \Lambda_{1}\cdots\Lambda_{n}(x)\quad\text{under }\Prob_{x}
\end{equation*}
for all $n\in\N$ and $x\in\R$.

\vspace{.1cm}
The following simple but crucial lemma is a consequence of Condition \eqref{eq:def2 AL}. Given a Lipschitz continuous function $f$, let $\Lip(f)$ be its Lipschitz constant. Note that $\Lip(\Lambda)=|\Aminus|\vee|\Aplus|$. Further putting $\Phi_{n}(x):=\Lip(\Lambda_{n})x+B_{n}$ for $n\in\N$, we introduce the \LIFS\  $(Z_{n})_{n\ge 0}$ and the associated "error term" $(Y_{n})_{n\ge 0}$ by setting $Y_{0}:=0$, $Z_{0}:=X_{0}$,
$$ Z_{n}\,:=\,\Lambda_{n}\cdots\Lambda_{1}(Z_{0})\quad\text{and}\quad Y_{n}\,:=\,\sum_{k=1}^{n}\Lip(\Lambda_{n}\ldots \Lambda_{k+1})B_{k} $$
for $n\ge 1$. The corresponding backward iterations are
\begin{equation}\label{eq:hatY_{n}}
\wh{Z}_{n}\,:=\,\Lambda_{1}\cdots\Lambda_{n}(Z_{0})\quad\text{and}\quad\wh{Y}_{n}\,:=\,\sum_{k=1}^{n}
\Lri_{k-1}B_{k},
\end{equation}
where $ \Lri_{k}\,:=\,\Lip(\Lambda_{1}\cdots\Lambda_{k})$
for $k\in\N$ and $\Lri_{0}:=1$.

\begin{Lemma}\label{lem:comparison with LIFS}
If Condition \eqref{eq:def2 AL} holds true, then
\begin{gather}
\sup_{x\in\R}|\Psi_{n}\cdots\Psi_{1}(x)-\Lambda_{n}\cdots\Lambda_{1}(x)|\ \le\ Y_{n},\label{eq:AL forward}\\
\sup_{x\in\R}|\Psi_{1}\cdots\Psi_{n}(x)-\Lambda_{1}\cdots\Lambda_{n}(x)|\ \le\ \wh{Y}_{n},\label{eq:AL backward}\\
\shortintertext{and in particular}
|X_{n}-Z_{n}|\ \le\ Y_{n}\quad\text{and}\quad |\wh{X}_{n}-\wh{Z}_{n}|\ \le\ \wh{Y}_{n}\label{eq:AL backward/forward}
\end{gather}
for all $n\in\N$.
\end{Lemma}

\begin{proof}
It suffices to prove \eqref{eq:AL backward} for which we use induction over $n$. Note that \eqref{eq:def2 AL} provides the assertion for $n=1$. Assuming the assertion be true for $n-1$ (inductive hypothesis), we infer for any $x\in\R$
\begin{align}
&|\Psi_{1}\cdots\Psi_{n}(x)-\Lambda_{1}\cdots\Lambda_{n}(x)|\nonumber\\
&\hspace{1cm}\le\ |\Psi_{1}\cdots\Psi_{n}(x)-\Lambda_{1}\cdots\Lambda_{n-1}(\Psi_{n}(x))|\nonumber\\
&\hspace{1.2cm}+\ |\Lambda_{1}\cdots\Lambda_{n-1}(\Psi_{n}(x))-\Lambda_{1}\cdots\Lambda_{n-1}(\Lambda_{n}(x))|\nonumber\\
&\hspace{1cm}\le\ \wh{Y}_{n-1}\,+\,
\Lri_{n-1}|\Psi_{n}(x)-\Lambda_{n}(x)|\label{eq:sharper estimate}\\
&\hspace{1cm}\le\ \wh{Y}_{n-1}\,+\,
\Lri_{n-1}B_{n}.\nonumber
\end{align}
Since, by \eqref{eq:hatY_{n}}, the last line equals $\wh{Y}_{n}$, the proof is complete.
\qed \end{proof}

\section{Transfer operators}\label{sec:transfer operators}

Aiming at the tail behavior of the stationary distributions of the given \ALIFS\ $(X_{n})_{n\ge 0}$ at $\pm\infty$, the Markov chain $(\xi_{n})_{n\ge 0}$ on the set $\bbS^{0}$ and its possibly  reduced transition matrix $P$ will play an important role. Similar to the work by Goldie \cite{Goldie:91} and Kesten \cite{Kesten:73}, our approach uses a linear approximation, here of $X_{n}$ by the \LIFS\ $Z_{n}=\Lambda_{n}\cdots\Lambda_{1}(Z_{0})$, see Lemma \ref{lem:comparison with LIFS}, and renewal-theoretic arguments after a suitable change of measure. The latter means to find a harmonic transform under which $\bbS^{0}$ becomes the proper state space of $(\xi_{n})_{n\ge 0}$, thus making absorption at 0 impossible if this state appears at all. For the case when $\log |Z_{n}|=S_{n}$ has i.i.d.~increments and thus forms an ordinary random walk on $\R$, this transform is usually obtained with the help of moment generating functions. The method has indeed been effectively employed in \cite{Goldie:91} and \cite{Mirek:11} in the study of asymptotically linear stochastic equations, see also \cite{BurDamMik:16}. In the present context, however, the sequence $(S_{n})_{n\ge 0}$ has increments whose distributions are modulated by a two-state Markov chain, and if this chain is irreducible, then $(\xi_{n},S_{n})_{n\ge 0}$ constitutes a genuine \MRW\ instead of an ordinary one. This in turn calls for the more advanced tool of so-called transfer operators, as in \cite{Kesten:73,AlsMen:12,BurDamGuiMen:14,GuivarchLePage:16} for the analysis of multidimensional problems and with a \MRW\ whose driving chain has state space $\mathbb{S}^{d-1}$, the $d$-dimensional unit sphere in $\R^{d}$ for some $d\ge 2$. Since $\bbS^{0}$ has only two elements, the transfer operators reduce here to fairly simple objects, namely $2\times 2$ matrices.

\subsection{The Cram\'er transform of $P$} \label{subsec:Cramer transform}

Recall from \eqref{eq:def P(theta)} the definition of the Cram\'er transform $P(\theta)$ of the transition matrix $P$ on its canonical domain $\bbD=\{\theta\ge 0:\Erw|\Aminus|^{\theta}+\Erw|\Aplus|^{\theta}<\infty\}$, in particular $P(0)=P$ and
\begin{equation}\label{eq: Ptheta vs Lambda}
p_{\delta\epsilon}(\theta)\ =\ \Erw|\Adelta|^{\theta}\1_{\{\sign(\Adelta)=\delta\epsilon\}}\ =\ \Erw_{\delta}\Big[e^{\theta S_{1}}\1_{\{\xi_{1}=\epsilon\}}\Big],
\end{equation}
the latter being a log-convex function on $\bbD$. We make the further assumption that
\begin{equation}\label{eq:theta_infty>0}
\theta_{\infty}\,:=\,\sup\bbD\ >\ 0.
\end{equation}
Let $\rho(\theta)$ be the dominant eigenvalue of $P(\theta)$ for $\theta\in\bbD$, explicitly given by
\begin{equation}\label{eq:rho(theta) explicit}\small
\rho(\theta)\ =\ \frac{\pmm(\theta)+\ppp(\theta)}{2}+\sqrt{\frac{(\pmm(\theta)-\ppp(\theta))^{2}}{4}+\pmp(\theta)\ppm(\theta)}\ >\ 0.
\end{equation}
The matrix $P(\theta)$ and its spectral radius  are strongly related to the product  $\Lambda_{n}\cdots\Lambda_{1}$ as confirmed by the subsequent lemma. Let $p_{\delta\epsilon}^{n}(\theta) = [P(\theta)^{n}]_{\delta,\epsilon}$ denote the entries of $P(\theta)^{n}$. We note that the $n$-step transition probability $p_{\delta\epsilon}^{n}(\theta) $ should not be confused with $p_{\delta\epsilon}(\theta)^{n}$, the $n^{th}$ power of $p_{\delta\epsilon}(\theta)$, which does also appear later.

\begin{Lemma} \label{lem:Ptheta vs Lambda}
For any $\theta\in \bbD$ and $\epsilon,\delta\in \bbS^{0}$
\begin{gather}
p_{\delta\epsilon}^{n}(\theta)\,=\,\Erw|\Lambda_{n}\cdots\Lambda_{1}(\delta)|^{\theta}\1_{\{\sign(\Lambda_{n}\cdots\Lambda_{1}(\delta))=\epsilon\}}\label{eq: Ptheta vs Lambda n}
\shortintertext{and}
\lim_{n\to\infty}\frac{1}{n}\log\Erw\,\Lip(\Lambda_{n}\cdots\Lambda_{1})^{\theta}\,=\,\log\rho(\theta).\label{eq:Lip vs rho}
\end{gather}
\end{Lemma}

\begin{proof}
The relation \eqref{eq: Ptheta vs Lambda n} can be proved by induction over $n$. For $n=1$, it holds true by \eqref{eq: Ptheta vs Lambda}, and for the inductive step we note that
\begin{align*}
p_{\delta\epsilon}^{n+1}(\theta)\,&=\,\sum_{s\in\bbS^{0}}p_{\delta s}(\theta)p_{ s\epsilon}^{n}(\theta)\\
&=\,\sum_{s\in\bbS^{0}}\Erw|\Lambda_{1}(\delta)|^{\theta}\1_{\{\sign(\Lambda_{1}(\delta))= s\}}|\Lambda_{n+1}\cdots\Lambda_2(s)|^{\theta}\1_{\{\sign(\Lambda_{n+1}\cdots\Lambda_2(s))=\epsilon\}}\\
&=\,\Erw|\Lambda_{n+1}\cdots\Lambda_{1}(\delta)|^{\theta}\1_{\{\sign(\Lambda_{n+1}\cdots\Lambda_{1}(\delta))=\epsilon\}}.
\end{align*}
In particular, the norm of $P^{n}(\theta)$ as an operator on $(\R^{2},|\cdot|_{\infty})$ equals
\begin{equation*}
\|P^{n}(\theta)\|_{\infty}\ =\ \max_{\delta}\left(p_{\delta +}^{n}(\theta)+p_{\delta -}^{n}(\theta)\right)\ =\ \max_\delta \Erw\left[|\Lambda_{n}\cdots\Lambda_{1}(\delta)|^{\theta}\right].
\end{equation*}
Hence
$$ \|P^{n}(\theta)\|_{\infty}\ \le\ \Erw\,\Lip(\Lambda_{n}\cdots\Lambda_{1})^{\theta}\ \le\  2\|P^{n}(\theta)\|_{\infty} $$
and Gelfand's formula yields \eqref{eq:Lip vs rho}.
\qed \end{proof}

For further discussion, we consider the three cases as introduced in Section~\ref{sec:main results} separately.

\vspace{.2cm}
\textit{\bfseries Case 1}. $\pmp\wedge\ppm>0$, i.e.~$P$ is irreducible.\\[1mm]
Then there are uniquely determined left and right nonnegative eigenvectors $u(\theta),v(\theta)$, respectively, satisfying
\begin{equation}\label{eq:normalization eigenvectors}
u(\theta)^{\top}v(\theta)\,=\,1\quad\text{and}\quad\uplus(\theta)+\uminus(\theta)\,=\,1.
\end{equation}
Moreover, $\rho(\theta)^{-1}P(\theta)$ has dominant eigenvalue 1 with the same eigenvectors and is therefore a \emph{quasistochastic} matrix in the sense of \cite[p.~360]{Alsmeyer:14}. This means that it is irreducible and nonnegative with maximal eigenvalue 1 and unique (up to positive scalars)  associated left and right eigenvectors. As shown in \cite[Section 2]{Alsmeyer:14}, $\rho(\theta)^{-1}P(\theta)$ can be transformed into a proper stochastic matrix $\wh{P}(\theta)$, namely
\begin{equation}\label{eq:Phat(theta)}
\wh{P}(\theta)\,:=\,\frac{1}{\rho(\theta)}D(\theta)^{-1}P(\theta)D(\theta)\ =\ \left(\frac{p_{\delta\epsilon}(\theta)v_{\epsilon}(\theta)}{\rho(\theta)v_{\delta}(\theta)}\right)_{\delta,\epsilon\in\bbS^{0}}
\end{equation}
with $D(\theta):=\diag(\vminus(\theta),\vplus(\theta))$. $\wh{P}(\theta)$ is irreducible with unique stationary distribution
\begin{equation}\label{eq:pihat(theta)}
\wh{\pi}(\theta)\ =\ D(\theta)u(\theta)\ =\ \big(\uminus(\theta)\vminus(\theta),\uplus(\theta)\vplus(\theta)\big)^{\top},
\end{equation}
which may also be written as
\begin{equation*}\small
\wh{\pi}(\theta)\ =\ \left(\frac{\ppm(\theta)\vminus(\theta)^{2}}{\ppm(\theta)\vminus(\theta)^{2}+\pmp(\theta)\vplus(\theta)^{2}},\frac{\pmp(\theta)\vplus(\theta)^{2}}{\ppm(\theta)\vminus(\theta)^{2}+\pmp(\theta)\vplus(\theta)^{2}}\right)^{\!\top}.
\end{equation*}
Therefore
\begin{equation*}
u(\theta)\ =\ \left(\frac{\ppm(\theta)\vminus(\theta)}{\ppm(\theta)\vminus(\theta)^{2}+\pmp(\theta)\vplus(\theta)^{2}},\frac{\pmp(\theta)\vplus(\theta)}{\ppm(\theta)\vminus(\theta)^{2}+\pmp(\theta)\vplus(\theta)^{2}}\right)^{\!\top},
\end{equation*}
which in combination with $\uminus(\theta)+\uplus(\theta)=1$ further entails
$$ \ppm(\theta)\vminus(\theta)(1-\vminus(\theta))\ +\ \pmp(\theta)\vplus(\theta)(1-\vplus(\theta))\ =\ 0. $$
By the ergodic theorem for positive recurrent Markov chains,
\begin{align}\label{eq:convergence Phat(theta) aperiodic}\small
\lim_{n\to\infty}\wh{P}(\theta)^{n}\ =\ \begin{pmatrix} \whpim(\theta) &\whpip(\theta)\\ \whpim(\theta) &\whpip(\theta)\end{pmatrix}\ =\ \begin{pmatrix} \uminus(\theta)\vminus(\theta) &\uplus(\theta)\vplus(\theta)\\ \uminus(\theta)\vminus(\theta) &\uplus(\theta)\vplus(\theta) \end{pmatrix}
\end{align}
if $\wh{P}(\theta)$ is aperiodic, while
\begin{align}\label{eq:convergence Phat(theta) periodic}\small
\wh{P}(\theta)^{n}\ =\
\begin{cases}
\hfill I\ =\ \begin{pmatrix} 1 &0\\ 0 &1\end{pmatrix},&\text{if $n$ is even},\\[4mm]
\wh{P}(\theta)\ =\ \begin{pmatrix} 0 &1\\ 1 &0\end{pmatrix},&\text{if $n$ is odd}
\end{cases}
\end{align}
for all $n\ge 1$ in the 2-periodic case $(\pmm+\ppp=0)$, noting in passing that all $\wh{P}(\theta)$ have the same period. Now it follows that, with $\1:=(1,1)^{\top}$,
\begin{align*}
\lim_{n\to\infty}\frac{1}{n}\log\left[D(\theta)\wh{P}(\theta)^{n}D(\theta)^{-1}\1\right]_{\delta}\ =\ 0\quad\text{for each }\delta \in  \bbS^{0},
\end{align*}
and this remains true with any other $w=(\wminus,\wplus)^{\top}\in\IRge^{2}\backslash\{(0,0)\}$ instead of $\1$. Since $\rho(\theta)^{-n}P(\theta)^{n}=D(\theta)\wh{P}(\theta)^{n}D(\theta)^{-1}$ by \eqref{eq:Phat(theta)}, we arrive at
\begin{align}\label{eq:limit form of rho(theta)}
\log\rho(\theta)\ =\ \lim_{n\to\infty}\frac{1}{n}\log [P(\theta)^{n}w]_{\delta}.
\end{align}
for each $\delta \in  \bbS^{0}$ and each $w=(\wminus,\wplus)^{\top}\in\IRge^{2}\backslash\{(0,0)\}$ which will be utilized in the proof of the subsequent lemma.

\begin{Lemma}\label{lem:properties rho(theta) case 1}
The function $\bbD\ni\theta\mapsto\log\rho(\theta)$ is continuous, convex, and on $(0,\theta_{\infty})$ also smooth (i.e.~infinitely often differentiable). Moreover,
\begin{equation}\label{eq:rho(theta) limit}
\log\rho(\theta)\ =\ \lim_{n\to\infty}\frac{1}{n}\,\log\Erw_{\pm}|\Lambda_{n}\cdots\Lambda_{1}(X_{0})|^{\theta}.
\end{equation}
\end{Lemma}

\begin{proof}
Since the components of $P(\theta)$ are continuous functions on $\bbD$ and even smooth on the interior of this set, the same properties hold for $\log\rho(\theta)$ because, by irreducibility,  $\ppm(\theta)\pmp(\theta)>0$ for all $\theta\in \bbD$ and thus $\rho(\theta)>0$, and by \eqref{eq:rho(theta) explicit}. The log-convexity of $\rho(\theta)$ is a direct consequence of  \eqref{eq:rho(theta) limit} (see also \cite[Prop.~1 and Cor.~2]{Saporta:05}). Finally, in order to obtain \eqref{eq:rho(theta) limit}, we can argue as follows after the observation that
\begin{align*}
P(\theta)\ =\ \left(p_{\delta\epsilon}\,\Erw\left[|\Adelta|^{\theta}\Big|\sign(\Adelta)\delta=\epsilon\right]\right)_{\delta,\epsilon\in\bbS^{0}}.
\end{align*}
By Lemma \ref{lem:Ptheta vs Lambda},
\begin{align*}
\Erw_{\delta}&|\Lambda_{n}\cdots\Lambda_{1}(X_{0})|^{\theta}=\
\sum_{s\in\bbS^{0}}p^n_{\delta s}(\theta)\
=\ [P(\theta)^{n}\1 ]_{\delta}\ 
\end{align*}
for any $\delta\in\bbS^{0}$. Taking the logarithm on both sides and dividing by $n$, \eqref{eq:rho(theta) limit} follows by means of \eqref{eq:limit form of rho(theta)}.
\qed
\end{proof}

\vspace{.2cm}
\textit{\bfseries Case 2}. $\pmp>0$ and $\ppm=0$ [the case when $\ppm>0$ and $\pmp=0$ can naturally be treated analogously].\\[1mm]
Then $P(\theta)$ is upper triangular for any $\theta\in\bbD$ with eigenvalues $\pmm(\theta),\ppp(\theta)$, giving $\rho(\theta)=\pmm(\theta)\vee\ppp(\theta)$. As a direct consequence, $\rho(\theta)$ is continuous and log-convex as the maximum of two such functions. It is also smooth at any $\theta$ with $\pmm(\theta)\ne\ppp(\theta)$, but may not be so if $\pmm(\theta)=\ppp(\theta)$.

\vspace{.2cm}
\textit{\bfseries Case 2A}. $\ppp(\theta)>\pmm(\theta)$.\\[1mm]
Then the left and right eigenvectors $u(\theta),v(\theta)$ satisfying \eqref{eq:normalization eigenvectors} are
\begin{gather}\label{eq:u(theta),v(theta) case 2A}
u(\theta)\,=\,e_{2}\,:=\,(0,1)^{\top}\quad\text{and}\quad v(\theta)\,=\,\left(\frac{\pmp(\theta)}{\ppp(\theta)-\pmm(\theta)},1\right)^{\top}.
\end{gather}
The matrix $\wh{P}(\theta)$, defined by \eqref{eq:Phat(theta)} and here no longer irreducible, equals
\begin{equation}\label{eq:Phat(theta) case 2A}
\wh{P}(\theta)\ =\ {\small\begin{pmatrix} \displaystyle\frac{\pmm(\theta)}{\ppp(\theta)} &\displaystyle 1-\frac{\pmm(\theta)}{\ppp(\theta)}\\[4mm] 0 & 1\end{pmatrix}},
\end{equation}
with unique stationary distribution $\wh{\pi}(\theta)=e_{2}$. Now it is readily checked that \eqref{eq:convergence Phat(theta) aperiodic} as well as \eqref{eq:rho(theta) limit} from Lemma \ref{lem:properties rho(theta) case 1} remain valid.

\vspace{.2cm}
\textit{\bfseries Case 2B}. $\pmm(\theta)>\ppp(\theta)$.\\[1mm]
Then the left and right eigenvectors $u(\theta),v(\theta)$ satisfying \eqref{eq:normalization eigenvectors} are
\begin{gather*}
u(\theta)\,=\,\left(\frac{\pmm(\theta)-\ppp(\theta)}{\pmm(\theta)+\pmp(\theta)-\ppp(\theta)},\frac{\pmp(\theta)}{\pmm(\theta)+\pmp(\theta)-\ppp(\theta)}\right)^{\top}
\shortintertext{and}
v(\theta)\,=\,\left(\frac{\pmm(\theta)+\pmp(\theta)-\ppp(\theta)}{\pmm(\theta)-\ppp(\theta)},0\right)^{\top},
\end{gather*}
but $\wh{P}(\theta)$ cannot be defined because $D(\theta)$ is not invertible. As $P(\theta)$ is upper triangular in the considered case, the probability that the chain $(\xi_n)_{n\ge 0}$, when starting in state $-$, remains there $n$ times equals $\pmm^{n}(\theta)=\pmm(\theta)^{n}$ for each $n\in\N$. Using this and recalling \eqref{eq: Ptheta vs Lambda n}, we see that
\begin{align}
\begin{split}\label{eq:rho(theta) case 2B}
\log\rho(\theta)\ &=\ \lim_{n\to\infty}\frac{1}{n}\,\log\pmm(\theta)^{n}\\
&=\ \lim_{n\to\infty}\frac{1}{n}\,\log\Eminus|\Lambda_{n}\cdots\Lambda_{1}(X_{0})|^{\theta}\1_{\{\xi_{1}=\ldots=\xi_{n}=-1\}}\\
&=\ \lim_{n\to\infty}\frac{1}{n}\,\log\Eminus|\Lambda_{n}\cdots\Lambda_{1}(X_{0})|^{\theta}\1_{\{\xi_{n}=-1\}}
\end{split}
\end{align}
holds instead of \eqref{eq:rho(theta) limit}.

\vspace{.2cm}
\textit{\bfseries Case 2C}. $\pmm(\theta)=\ppp(\theta)$.\\[1mm]
This is the boundary case where left and right eigenvectors equal $u(\theta)=e_{2}$ and $v(\theta)=e_{1}:=(1,0)^{\top}$, respectively, and are thus orthogonal. Furthermore, \eqref{eq:rho(theta) case 2B} as well as
\begin{equation}\label{eq:rho(theta) case 2C}
\log\rho(\theta)\ =\ \lim_{n\to\infty}\frac{1}{n}\,\log\Eplus \Lambda_{n}\cdots\Lambda_{1}(X_{0})
\end{equation}
hold true.

\subsection{The measure change}\label{subsec:measure change}

Let $\cF$ be the $\sigma$-field generated by $(\xi_{n},S_{n})_{n\ge 0}$ and recall that $\zeta_{1},\zeta_{2},\ldots$ denote the increments of the $S_{n}$.

\vspace{.2cm}
\textit{\bfseries Case 1}. $\pmp\wedge\ppm>0$.\\[1mm]
We make the additional assumption that
\begin{equation}\label{eq:existence kappa}
\rho(\kappa)=1\quad\text{for some }\kappa\in (0,\theta_{\infty}]
\end{equation}
and note that $\kappa$ is unique because $\rho(\theta)$ is convex, $\rho(0)\le 1$, and $\rho(\theta)<1$ for $\theta$ in a right neighborhood of $0$ (Lemma \ref{lem:rho(s)<1} in the Appendix). In particular, the monotonicity of $\rho'(\theta)$ for $\theta\in (0,\theta_{\infty})$
entails $\rho'(\kappa)=\lim_{\theta\uparrow\kappa}\rho'(\theta)>0$, a fact to be used below in the proof of Lemma \ref{lem:stationary drift}. Furthermore, $P(\kappa)$ is quasistochastic and the associated transformation $\wh{P}(\kappa)$, defined by
\begin{align*}
\wh{P}(\kappa)\ =\ \begin{pmatrix} \pmm(\kappa) &\frac{\pmp(\kappa)\vplus(\kappa)}{\vminus(\kappa)}\\ \frac{\ppm(\kappa)\vminus(\kappa)}{\vplus(\kappa)} &\ppp(\kappa) \end{pmatrix}\ =\ \left(\frac{v_{\epsilon}(\kappa)}{v_{\delta}(\kappa)}\,\Erw_{\delta}\big[e^{\kappa S_{1}}\1_{\{\xi_{1}=\epsilon\}}\big]\right)_{\delta,\epsilon\in\bbS^{0}}
\end{align*}
(see \eqref{eq:Phat(theta)}), is an irreducible Markov transition matrix with unique stationary distribution $\wh{\pi}:=\wh{\pi}(\kappa)$ given by \eqref{eq:pihat(theta)}. Since the sequence $(e^{\theta S_n} v_{\xi_n}\rho(\theta)^{-n})_{n\ge 0}$ forms a positive martingale, it allows to define a new probability measure $\Prob_{\delta}^{(\theta)}$  on the path space $\cF=(\bbS^{0}\times\R)^{\N_{0}}$ of $(\xi_{n},\zeta_{n})_{n\ge 0}$ by
\begin{equation}\label{eq:def erw}
\Erw_{\delta}^{(\theta)}f((\xi_{0},\zeta_{0}),\ldots,(\xi_{n},\zeta_{n}))\ =\ \frac{\Erw_{\delta}\big[e^{\theta S_{n}}v_{\xi_{n}}(\theta)f((\xi_{1},\zeta_{1}),\ldots,(\xi_{n},\zeta_{n}))\big]}{v_{\delta}(\theta)\rho(\theta)^{n}}
\end{equation}
for all $n\in\N_{0}$ and bounded measurable $f:(\bbS^{0}\times\R)^{n}\to\R$. As shown in the next lemma, $(\xi_{n},\zeta_{n})_{n\ge 0}$ is still a Markov chain on $\bbS^{0}\times\R$ under $\Prob_{\delta}^{(\theta)}$, but with new transition kernel $Q_{\theta}$ given by
\begin{equation}\label{eq:Qtheta f}
Q_{\theta}f(\delta,x)\,:=\,\frac{1}{v_{\delta}(\theta)\rho(\theta)}\,\Erw_{\delta}\Big[e^{\theta\zeta_{1}}v_{\xi_{1}}(\theta)f(\xi_{1},\zeta_{1})\Big]
\end{equation}
for bounded functions $f:\bbS^{0}\times\R\to\R$. Hence, the conditional law of $(\xi_{1},\zeta_{1})$ given initial state $(\delta,x)$ depends only on $\delta$ but not on $x$.

\begin{Lemma}\label{lem:Pdelta^kappa case 1}
For $\theta\in\bbD$ and $\delta\in\bbS^{0}$, the following assertions hold under the new probability measure $\Prob_{\delta}^{(\theta)}$:
\begin{itemize}\itemsep2pt
\item[(a)] $(\xi_{n},\zeta_{n})_{n\ge 0}$ is a Markov chain with transition operator $Q_{\theta}$ and $\xi_{0}=\delta$.
\item[(b)] $(\xi_{n})_{n\ge 0}$ is an irreducible Markov chain on $\bbS^{0}$ with transition matrix $\wh{P}(\theta)$ and unique stationary distribution $\wh{\pi}(\theta)$.
\item[(c)] $(\xi_{n},S_{n})_{n\ge 0}$ is a \MRW\ with driving chain $(\xi_{n})_{n\ge 0}$ and $S_{0}=0$.
\end{itemize}
\end{Lemma}

\begin{proof}
(a) It suffices to note that, for arbitrary bounded measurable $f,g$ with obvious domains and $n\in\N$,
\begin{align*}
&\Erw_{\delta}^{(\theta)}[f((\xi_{1},\zeta_{1}),\ldots,(\xi_{n},\zeta_{n}))g(\xi_{n+1},\zeta_{n+1})]\\
&=\ \frac{\Erw_{\delta}\big[e^{\theta S_{n}}f((\xi_{1},\zeta_{1}),\ldots,(\xi_{n},\zeta_{n}))e^{\theta\zeta_{n+1}}v_{\xi_{n+1}}(\theta)g(\xi_{n+1},\zeta_{n+1})\big]}{v_{\delta}(\theta)\rho(\theta)^{n+1}}\\
&=\ \frac{\Erw_{\delta}\big[e^{\theta S_{n}}f((\xi_{1},\zeta_{1}),\ldots,(\xi_{n},\zeta_{n}))v_{\xi_{n}}(\theta)Q_{\theta}g(\xi_{n},\zeta_{n})\big]}{v_{\delta}(\theta)\rho(\theta)^{n}}\\
&=\ \Erw_{\delta}^{(\theta)}[f((\xi_{1},\zeta_{1}),\ldots,(\xi_{n},\zeta_{n}))Q_{\theta}g(\xi_{n},\zeta_{n})\big]
\end{align*}
holds true, implying
$$ \Erw_{\delta}^{(\theta)}\big[g(\xi_{n+1},\zeta_{n+1})\big|\cF_{n}\big]\ =\ Q_{\theta}g(\xi_{n},\zeta_{n})\quad\text{a.s.} $$
for $\cF_{n}:=\sigma((\xi_{k},\zeta_{k}),0\le k\le n)$. 

\vspace{.1cm}
(b) and (c) are  direct consequences of (a).\qed 
\end{proof}

Recall that $\wh{\pi}(\theta)=\big(\uminus(\theta)\vminus(\theta),\uplus(\theta)\vplus(\theta)\big)^{\top}$ defined in \eqref{eq:pihat(theta)} equals the stationary distribution of $\wh{P}(\theta)$ and thus of $(\xi_{n})_{n\ge 0}$ under $\Prob_{\delta}^{(\theta)}$, $\delta\in\bbS^{0}$.

\begin{Lemma}\label{lem:stationary drift}
Under $\Prob_{\wh{\pi}(\theta)}^{(\theta)}$, the \MRW\ $(\xi_{n},S_{n})_{n\ge 0}$ has drift
\begin{equation}\label{eq:stationary drift}
\Erw_{\wh{\pi}(\theta)}^{(\theta)}S_{1}\ =\ \frac{\rho'(\theta)}{\rho(\theta)}
\end{equation}
for any $\theta\in\bbD$ and is finite for $\theta\in (0,\theta_{\infty})$, the interior of $\bbD$.
The drift under $\Prob_{\wh{\pi}(\kappa)}^{(\kappa)}$ is positive, but possibly infinite if $\kappa=\theta_{\infty}$. In the latter case, \eqref{eq:stationary drift} still holds with $\rho'(\kappa):=\lim_{\theta\uparrow\kappa}\rho'(\theta)$.
\end{Lemma}

\begin{proof}
Recalling that $u(\theta)^{\top}v(\theta)=1$, thus $u(\theta)^{\top}P(\theta)v(\theta)=\rho(\theta)$ for $\theta\in\bbD$, it follows by differentiation that, for $\theta\in (0,\theta_{\infty})$,
\begin{align*}
\rho'(\theta)\ &=\ \frac{d}{d\theta}\left[u(\theta)^{\top}P(\theta)v(\theta)\right]\\
&=\ u'(\theta)^{\top}P(\theta)v(\theta)\,+\,u(\theta)^{\top}P'(\theta)v(\theta)\,+\,u(\theta)^{\top}P(\theta)v'(\theta)\\
&=\ \rho(\theta)\Big(u'(\theta)^{\top}v(\theta)+u(\theta)^{\top}v'(\theta)\Big)\,+\,u(\theta)^{\top}P'(\theta)v(\theta)\\
&=\ \rho(\theta)\frac{d}{d\theta}\left[u(\theta)^{\top}v(\theta)\right]\,+\,u(\theta)^{\top}P'(\theta)v(\theta)\\
&=\ u(\theta)^{\top}P'(\theta)v(\theta),
\end{align*}
where $P'(\theta)=(p_{\delta\epsilon}'(\theta))_{\delta,\epsilon\in\bbS^{0}}$ denotes the (componentwise) derivative of $P(\theta)$. Since, on the other hand,
\begin{align*}
\Erw_{\wh{\pi}(\theta)}^{(\theta)}S_{1}\ &=\ \frac{1}{\rho(\theta)}\sum_{\delta\in\bbS^{0}}\frac{\wh{\pi}_{\delta}(\theta)}{v_{\delta}(\theta)}\,\Erw_{\delta}\big[e^{\theta S_{1}}S_{1}v_{\xi_{1}}(\theta)\big]\\
&=\ \frac{1}{\rho(\theta)}\sum_{\delta,\epsilon\in\bbS^{0}}u_{\delta}(\theta)\,\Erw_{\delta}\big[e^{\theta S_{1}}S_{1}\1_{\{\xi_{1}=\epsilon\}}\big]\,v_{\epsilon}(\theta)\\
&=\ \frac{1}{\rho(\theta)}\sum_{\delta,\epsilon\in\bbS^{0}}u_{\delta}(\theta)p_{\delta\epsilon}'(\theta)v_{\epsilon}(\theta)\\
&=\ \frac{u(\theta)^{\top}P'(\theta)v(\theta)}{\rho(\theta)},
\end{align*}
we see that \eqref{eq:stationary drift} holds. We have already pointed out above (see after \eqref{eq:existence kappa}) that $\rho'(\kappa)=\lim_{\theta\uparrow\kappa}\rho'(\theta)>0$, and we add that $\rho'(\kappa)$ is finite if $\kappa<\theta_{\infty}$, but may be infinite if $\kappa$ equals the upper boundary of $\bbD$.
\qed \end{proof}

\begin{Rem}\rm
The following argument shows that under Condition \eqref{eq:moments AlogA B} of Thm.~\ref{thm:main case 1}, the finiteness of $\Erw_{\wh{\pi}(\kappa)}^{(\kappa)}S_{1}=\rho'(\kappa)$ is always guaranteed, even if $\kappa=\theta_{\infty}$. Indeed, recalling $u(\kappa)=\wh{\pi}(\kappa)$ and $v(\kappa)=(1,1)^{\top}$, it follows that
\begin{align*}
\rho'(\kappa)\ &=\ \Erw_{\wh{\pi}(\kappa)}^{(\kappa)}S_{1}\ =\ \sum_{\delta,\epsilon\in\bbS^{0}}\wh{\pi}_{\delta}(\kappa)\,\Erw_{\delta}\big[\log|\Adelta|\,|\Adelta|^{\kappa}\1_{\{\sign(\Adelta)\delta
=\epsilon\}}\big]\nonumber\\
&=\ \whpim(\kappa)\,\Erw|\Aminus|^{\kappa}\log|\Aminus|\,+\,\whpip(\kappa)\,\Erw|\Aplus|^{\kappa}\log|\Aplus|
\end{align*}
and then that $\Erw_{\wh{\pi}(\kappa)}^{(\kappa)}S_{1}<\infty$ holds iff $\Erw|\Aminus|^{\kappa}\log|\Aminus|+\Erw|\Aplus|^{\kappa}\log|\Aplus|<\infty$,
because $\wh{\pi}(\kappa)$ has positive entries.
\end{Rem}

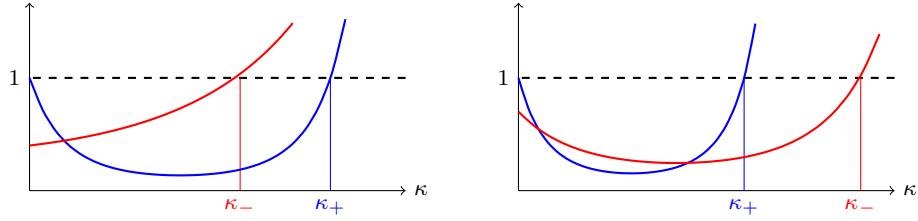
\begin{figure}[t]
\centering
\begin{tikzpicture}
 \draw[->] (0,0) -- (5,0) node[right] {$\kappa$};
 \draw[->] (0,0) -- (0,2.5) node[above] {$ $};
 \draw[-,thick,dashed] (5,1.5) -- (0,1.5) node[left]{${1}$};
 \draw[-,red] (2.8,1.5) -- (2.8,0) node[below]{\color{red} $\kappa_{-}$};
 \draw[-,blue] (4,1.5) -- (4,0) node[below]{\color{blue} $\kappa_{+}$};
 \draw[scale=1,domain=0:4.2,smooth,variable=\x,blue,thick] plot ({\x},
 {1.5*exp(-2.*\x+(\x*\x)/2)});
 \draw[scale=1,domain=0:3.5,smooth,variable=\x,red,thick] plot ({\x},
 {.6*exp(.2*\x+(\x*\x)/20)});
 \draw[->] (6.5,0) -- (11.5,0) node[right] {$\kappa$};
 \draw[->] (6.5,0) -- (6.5,2.5) node[above] {$ $};
 \draw[-,thick,dashed] (11.5,1.5) -- (6.5,1.5) node[left]{${1}$};
 \draw[-,red] (11.05,1.5) -- (11.05,0) node[below]{\color{red} $\kappa_{-}$};
 \draw[-,blue] (9.5,1.5) -- (9.5,0) node[below]{\color{blue} $\kappa_{+}$};
 \draw[domain=0:3.15,smooth,variable=\x,blue,thick] plot ({6.5+\x},
 {1.5*exp(-2.5*\x+(\x*\x)/1.2)});
 \draw[domain=0:4.8,smooth,variable=\x,red,thick] plot ({6.5+\x},
 {1.05*exp(-\x+(\x*\x)/4.2)});
\end{tikzpicture}
\caption{The two functions $\ppp(\theta)$ (blue) and $\pmm(\theta)$ (red) and the two possible constellations for $\kappa_{-}$ and $\kappa_{+}$ if both values exist. }\label{fig2}
\end{figure}

\vspace{.1cm}
\textit{\bfseries Case 2}. $\pmp>0$ and $\ppm=0$.\\[1mm]
Recall that $P(\theta)$ is upper triangular and $\rho(\theta)=\pmm(\theta)\vee\ppp(\theta)$ for any positive $\theta\in\bbD$.

\vspace{.2cm}
\textit{\bfseries Case 2A}. $\ppp(\theta)>\pmm(\theta)$.\\[1mm]
Then $\rho(\theta)=\ppp(\theta)$ and the following lemma is almost identical with Lemma \ref{lem:Pdelta^kappa case 1}, the only difference being that $(\xi_{n})_{n\ge 0}$ is obviously no longer irreducible. It is therefore stated without proof.

\begin{Lemma}\label{lem:Pdelta^kappa case2A}
For positive $\theta\in\bbD$ and $\delta\in\bbS^{0}$, define the probability measure $\Prob_{\delta}^{(\theta)}$ on $\cF$ as in Lemma \ref{lem:Pdelta^kappa case 1}. Then the following holds under each $\Prob_{\delta}^{(\theta)}$:
\begin{itemize}\itemsep2pt
\item[(a)] Lemma \ref{lem:Pdelta^kappa case 1}(a) and (c) remain valid with $\wh{P}(\theta),\,v(\theta)$ as stated in \eqref{eq:u(theta),v(theta) case 2A} and \eqref{eq:Phat(theta) case 2A}.
\item[(b)] State $-1$ is transient and $+1$ absorbing for $(\xi_{n})_{n\ge 0}$ on $\bbS^{0}$.
\end{itemize}
\end{Lemma}

The fact that state $+$ is absorbing for $(\xi_{n})_{n\ge 0}$ entails that $(S_{n})_{n\ge 0}$ forms an ordinary random walk under $\IPplus^{(\theta)}$ and
\begin{align*}
\Eplus ^{(\theta)}g(\zeta_{1},\ldots,\zeta_{n})\ =\ \Eplus \big[e^{\theta  S_{n}}g(\zeta_{1},\ldots,\zeta_{n})\big]/\ppp(\theta)^{n}
\end{align*}
for any bounded measurable $g:\R^{n}\to\R$. Since $\wh{\pi}(\theta)=e_{2}$, the stationary drift of $(S_{n})_{n\ge 0}$ is given by
\begin{equation}\label{eq:drift positive case 2A}
\Erw_{\wh{\pi}(\theta)}^{(\theta)}S_{1}\ =\ \Eplus ^{(\theta)}S_{1}\ =\ \ppp'(\theta).
\end{equation}
and thus positive if $\ppp'(\theta)>0$.

\vspace{.2cm}
\textit{\bfseries Case 2B and 2C}. $\pmm(\theta)\ge\ppp(\theta)$.\\[1mm]
In the remaining two subcases 2B and 2C, which can be treated together, the chain $\Xi$ is constant under each $\Prob_{\delta}^{(\theta)}$ as defined below and thus $(S_{n})_{n\ge 0}$ is an ordinary random walk under these probability measures. The result is summarized in the subsequent lemma which we state again without proof. 

\begin{Lemma}\label{lem:Pdelta^kappa case 2BC}
For positive $\theta\in\bbD$ and $\delta\in\bbS^{0}$, define $\Prob_{\delta}^{(\theta)}$ on $\cF$ by
\begin{align*}
\Erw_{\delta}^{(\theta)}f((\xi_{1},\zeta_{1}),\ldots,(\xi_{n},\zeta_{n}))\ =\ \frac{\Erw_{\delta}\big[e^{\theta S_{n}}f((\xi_{1},\zeta_{1}),\ldots,(\xi_{n},\zeta_{n}))\1_{\{\xi_{n}=\delta\}}\big]}{p_{\delta\delta}(\theta)^{n}}
\end{align*}
for all $n\in\N$ and bounded measurable $f:(\bbS^{0}\times\R)^{n}\to\R$. Then the following holds true under each $\Prob_{\delta}^{(\theta)}$:
\begin{itemize}\itemsep2pt
\item[(a)] $\xi_{n}=\delta$ a.s.~for all $n\ge 0$.
\item[(b)] $(S_{n})_{n\ge 0}$ is an ordinary random walk with $S_{0}=0$ and drift
\begin{equation*}
\Erw_{\delta}^{(\theta)}S_{1}\ =\ \frac{p_{\delta\delta}'(\theta)}{p_{\delta\delta}(\theta)}
\end{equation*}
which equals $\pmm'(\theta)$ and is positive if $\delta=-1$ and $\pmm(\theta)=1$.
\end{itemize}
\end{Lemma}

For the very last assertion, we note that $\pmm(\theta)=1$, i.e. $\theta=\kappaminus$,
does indeed entail $\pmm'(\theta)>0$ because $\pmm(0)<1$ and $\pmm(\cdot)$ is a convex function on $\bbD$. If $\pmm(\theta)=\ppp(\theta)=1$ (Case 2C), then $\Eplus ^{(\theta)}S_{1}$ equals $\ppp'(\theta)$ and is also positive. However, positivity may fail otherwise.

\subsection{Renewal theorems}\label{subsec:MRT}

We are now ready to state the renewal theorems for the \MRW\ $(\xi_{n},S_{n})_{n\ge 0}$ with driving chain $\Xi=(\xi_{n})_{n\ge 0}$ needed to derive our main results stated in Section \ref{sec:main results}.

\vspace{.2cm}
\textit{\bfseries Case 1}. $\pmp\wedge\ppm>0$.\\[1mm]
We assume \eqref{eq:existence kappa}, thus $\rho(\kappa)=1$ for some $\kappa\in\bbD$.
By Lemma \ref{lem:Pdelta^kappa case 1}, $\Xi$ is an irreducible finite Markov chain under $\Prob_{\delta}^{(\kappa)}$. Therefore, the subsequent lemma for $(\xi_{n},S_{n})_{n\ge 0}$ follows from essentially any version of the Markov renewal theorem that has appeared in the literature, see e.g.~\cite{Jacod:71,Shurenkov:84,Athreya+et.al.:78a,Alsmeyer:94} and most recently \cite{Alsmeyer:14} where it is derived probabilistically from the classical Blackwell theorem. Only the usual lattice-type assumption requires a little more care:

\vspace{.1cm}
Following Shurenkov \cite{Shurenkov:84}, the \MRW\ $(\xi_{n},S_{n})_{n\ge 0}$ is called nonarithmetic with respect to the probability measures $\Prob_{\delta}^{(\kappa)}$, $\delta\in\bbS^{0}$, if
\begin{equation*}
\Prob_{\wh{\pi}(\kappa)}^{(\kappa)}(S_{1}\in g(\xi_{1})-g(\xi_{0})+d\Z)\ <\ 1
\end{equation*}
for all $g:\bbS^{0}\to\R$ and $d\in (0,\infty)$, where $d\Z:=\{0,\pm d,\pm 2d,\ldots\}$. Equivalently, see \cite[Lemma 3.3]{Alsmeyer:94}, $\Prob_{\delta}^{(\kappa)}(S_{\tau(\delta)}\in\cdot)$ is nonarithmetic in the usual sense for each $\delta\in\bbS^{0}$, where
$$ \tau(\delta)\,:=\,\inf\{n\ge 1:\xi_{n}=\delta\}. $$
Since $S_{1}=\log|{}^{\xi_{0}}\!A|$, we see that, putting $a_{\pm}:=g(\pm 1)$, \eqref{eq:log A nonlattice} in Thm.~\ref{thm:main case 1} is indeed equivalent to $(\xi_{n},S_{n})_{n\ge 0}$ being nonarithmetic.

\begin{Lemma}\label{lem:RThm case 1}
Given the stated assumptions, suppose further that the \MRW\ $(\xi_{n},S_{n})_{n\ge 0}$ is nonarithmetic. Then
\begin{equation}\label{eq:KRT case 1}
\sum_{n\ge 0}\Erw_{\epsilon}^{(\kappa)}g(\xi_{n},t-S_{n})\ \xrightarrow{t\to\infty}\ \frac{1}{\rho'(\kappa)}\sum_{\delta\in\bbS^{0}}\wh{\pi}_{\delta}(\kappa)\int_{\R}g(\delta,x)\ dx
\end{equation}
for any $\epsilon\in\bbS^{0}$ and any measurable $g:\bbS^{0}\times\R\to\R$ such that $x\mapsto g(\delta,x)$ is directly Riemann integrable (dRi) for each $\delta\in\bbS^{0}$.
\end{Lemma}

As a direct consequence of this lemma, we obtain a key renewal theorem for the two-sided \MRW\ $(\xi_{n},S_{n})_{n\ge 0}$ under the original probability measures $\Prob_{\epsilon}$.

\begin{Prop}\label{prop:RThm case 1}
Under the assumptions of the previous lemma,
\begin{equation}\label{eq:KRT2 case 1}
e^{\kappa t}\sum_{n\ge 0}\Erw_{\epsilon}g(\xi_{n},t-S_{n})\ \xrightarrow{t\to\infty}\ \frac{v_{\epsilon}(\kappa)}{\rho'(\kappa)}\sum_{\delta\in\bbS^{0}}u_{\delta}(\kappa)\int_{\R}e^{\kappa x}g(\delta,x)\ dx
\end{equation}
for any $\epsilon\in\bbS^{0}$ and any measurable $g:\bbS^{0}\times\R\to\R$ such that $x\mapsto e^{\kappa x}g(\delta,x)$ is dRi for each $\delta\in\bbS^{0}$.
\end{Prop}

\begin{proof}
In view of the previous lemma, \eqref{eq:KRT2 case 1} follows from
\begin{align*}
e^{\kappa t}\sum_{n\ge 0}\Erw_{\epsilon}g(\xi_{n},t-S_{n})\ &=\ \sum_{n\ge 0}\Erw_{\epsilon}\left[\frac{e^{\kappa(t-S_{n})}}{v_{\xi_{n}}(\kappa)}g(\xi_{n},t-S_{n})e^{\kappa S_{n}}v_{\xi_{n}}(\kappa)\right]\\
&=\ \sum_{n\ge 0}v_{\epsilon}(\kappa)\Erw_{\epsilon}^{(\kappa)}\left[\frac{e^{\kappa(t-S_{n})}}{v_{\xi_{n}}(\kappa)}g(\xi_{n},t-S_{n})\right]\\
&\xrightarrow{t\to\infty}\ \frac{v_{\epsilon}(\kappa)}{\rho'(\kappa)}\sum_{\delta\in\bbS^{0}}\wh{\pi}_{\delta}(\kappa)\int_{\R}\frac{e^{\kappa x}}{v_{\delta}(\kappa)}g(\delta,x)\ dx
\end{align*}
when recalling that $\wh{\pi}_{\delta}(\kappa)=u_{\delta}(\kappa)v_{\delta}(\kappa)$.
\qed \end{proof}

\vspace{.2cm}
\textit{\bfseries Case 2}. $\ppm=0<\pmp$.\\[1mm]
Recall from before Thm.~\ref{thm:main case 2} that $\kappaminus$ and $\kappaplus$ are the unique positive numbers satisfying $\pmm(\kappaminus)=1$ and $\ppp(\kappaplus)=1$, respectively, provided that these numbers exist. Otherwise, we put $\kappa_{\pm}:=\infty$ but make the additional assumption that at least one of them is finite, thus
\begin{equation}\label{eq:kappa finite case 2}
\kappa\,:=\,\kappaminus\wedge\kappaplus\,<\,\infty.
\end{equation}

\vspace{.2cm}
\textit{\bfseries Case 2A}. $\ppm=0<\pmp$, \eqref{eq:kappa finite case 2} holds and $\ppp(\kappa)=1>\pmm(\kappa)$.\\[1mm]

In this case $\kappa=\kappaplus$ and $-1$ is a transient state for the chain $\Xi$. Therefore, $\Xi$ will eventually reach $+1$ or $0$. The state $0$ is absorbing and so $S_{n}=-\infty$ whenever $\xi_{n}=0$.  Put $\overline{\tau}:=\tau(0)\wedge\tau(1)$ and write $\tau$ as shorthand for $\tau(1)$. Defining
\begin{align*}\IUminus(D)\ :&=\ \sum_{n\ge 0}\IPminus[S_{n}\in D, \xi_{n}=-1]\\
&=\sum_{n\ge 0}\pmm^{n}\,\IPminus[S_{n}\in D|\xi_{n}=-1] \ =\ \Eminus\Bigg[\sum_{n=0}^{\overline{\tau}-1}\1_{\{S_{n}\in D\}}\Bigg]
\end{align*}
for measurable $D\subset\R$, we obtain the renewal measure associated with a defective distribution, namely $\pmm\,\IPminus[S_{1}\in\cdot|\xi_{1}=-1]$. On the other hand, $(S_{\tau+n})_{n\ge 0}$ forms an ordinary random walk on $\R\cup\{-\infty\}$ under both $\IPminus$ and $\IPplus$, with initial value $S_{\tau}$ and  increment distribution $\IPplus[\xi_{1}=1,S_{1}\in\cdot]=\IPplus[S_{1}\in\cdot]$. In particular
\begin{align*}
\sum_{n\ge 0}\IPminus[S_{n}\in D, \xi_{n}=+1]\,&=\,\Eminus\!\Bigg[\sum_{n\ge\tau}\1_{\{S_{n}\in D\}}\Bigg]\,=\,\Eminus\left[ \IUplus(D-S_{\tau}) 1_{\{\tau<\infty\}}\right]
\end{align*}
where
$$ \IUplus(D)\,:=\,\Eplus\Bigg[\sum_{n\ge 0}\1_{\{S_{n}\in D\}}\Bigg]. $$

After these observations, the subsequent result is proved by a standard combination of measure change with classical renewal theory. For an arbitrary function $g:\R\to\R$ and $\theta\in\R$, we put $g_{\theta}(x):=e^{\theta x}g(x)$ and stipulate as usual $\infty^{-1}:=0$. We further denote by $\cR_{\theta}$ the space of functions $g:\R\to\R$ such that $g_{\theta}$ is dRi.

\begin{Prop}\label{prop:RThm case 2A}
Under the stated assumptions, the following assertions hold true (with $\kappa=\kappaplus$):
\begin{gather}\small
e^{\theta t}\,g*\IUminus(t)\ =\ e^{\theta t}\,\Eminus\Bigg[\sum_{n=0}^{\tau-1}g(t-S_{n})\Bigg]\ \xrightarrow{|t|\to\infty}\ 0\label{eq:KRT case 2A-}
\end{gather}
for each $\theta>0$ such that $\pmm(\theta)<1$ and $g\in\cR_{\theta}$. If $\kappaminus<\infty$, $g\in\cR_{\kappaminus}$ and $\IPminus[S_{1}\in\cdot|\xi_{1}=-1]$ is nonarithmetic, then in addition to \eqref{eq:KRT case 2A-}
\begin{equation}\label{eq2:KRT case 2A-}
e^{\kappaminus t}\,g*\IUminus(t)\ \xrightarrow{t\to\infty}\ \frac{1}{\pmmprime(\kappaminus)}\int_{\R}g_{\kappaminus}(x)\ dx.
\end{equation}
Furthermore, if $S_{1}$ is nonarithmetic under $\IPplus$, then
\begin{gather}\small
e^{\kappa t}\,\Eminus\left[\sum_{n\ge\tau}g(t-S_{n})\right]\ \xrightarrow{t\to\infty}\ \frac{\pmp(\kappa)}{(1-\pmm(\kappa))\pppprime(\kappa)}\int_{\R}g_{\kappa}(x)\ dx\label{eq:KRT case 2A-+}
\shortintertext{and}\small
e^{\kappa t}\,g*\IUplus(t)\ =\ e^{\kappa t}\,\Eplus\left[\sum_{n\ge 0}g(t-S_{n})\right]\ \xrightarrow{t\to\infty}\ \frac{1}{\pppprime(\kappa)}\int_{\R}g_{\kappa}(x)\ dx\label{eq:KRT case 2A+}
\end{gather}
for any function $g\in\cR_{\kappa}$.
\end{Prop}

\begin{proof}
Then
$$ \bUminus^{(\theta)}\,:=\,\Eminus\left[\sum_{n\ge 0}e^{\theta S_{n}}\1_{\{S_{n}\in\cdot\,,\,\xi_{n}=-1\}}\right]\ =\ \sum_{n\ge 0}\pmm(\theta)^{n}\,\IPminus^{(\theta)}[S_{1}\in\cdot|\xi_{1}=-1]^{*n} $$
defines a defective and thus finite ordinary renewal measure. As a consequence,
\begin{align*}
e^{\theta t}\,g*\IUminus(t)\ &=\ \Eminus\left[\sum_{n\ge 0}e^{\theta S_{n}}g_{\theta}(t-S_{n})\1_{\{\xi_{n}=-1\}}\right]\ =\ g_{\theta}*\bUminus^{(\theta)}(t)
\end{align*}
converges to 0 as $t\to\infty$ for any $g\in\cR_{\theta}$, which proves \eqref{eq:KRT case 2A-}.

\vspace{.1cm}
If $\kappaminus<\infty$, then $\pmm(\kappaminus)=1$ and therefore, by \eqref{eq:def P(theta)},
$$ \IPminus^{(\kappaminus)}[S_{1}\in D]\,:=\,\Eminus[e^{\kappaminus S_{1}}\1_{D}(S_{1})\1_{\{\xi_{1}=-1\}}] $$
for measurable $D\subset\R$ defines a probability measure. Moreover,
\begin{gather*}
e^{\kappaminus t}\,g*\IUminus(t)=\Eminus\left[\sum_{n=0}^{\infty}e^{\kappaminus S_{n}}g_{\kappaminus}(t-S_{n})\1_{\{\xi_{n}=-1\}}\right]=g_{\kappaminus}*\IUminus^{(\kappaminus)}(t),
\shortintertext{where}
\IUminus^{(\kappaminus)}\,:=\,\sum_{n\ge 0}\IPminus^{(\kappaminus)}[S_{1}\in\cdot]^{*n}.
\end{gather*}
Hence, \eqref{eq2:KRT case 2A-} follows by another appeal to the key renewal theorem.

\vspace{.1cm}
Turning to \eqref{eq:KRT case 2A-+} and \eqref{eq:KRT case 2A+} which can be shown together, we note that
\begin{align*}
e^{\kappa t}\,\Eminus\Bigg[\sum_{n\ge\tau}g(t-S_{n})\Bigg]\ &=\ \Eminus\Bigg[\sum_{n\ge\tau}e^{\kappa S_{n}}g_{\kappa}(t-S_{n})\Bigg]\\
&=\ \Eminus\Big[ e^{\kappa S_\tau}\,g_{\kappa}*\IUplus^{(\kappa)}(t-S_\tau)\1_{\{\tau<\infty\}}\Big]
\end{align*}
where
$$ \IUplus^{(\kappa)}\,:=\,\sum_{n\ge 0}\IPplus^{(\kappa)}[S_{n}\in\cdot,\xi_{n}=1]\ =\ \sum_{n\ge 0}\IPplus^{(\kappa)}[S_{n}\in\cdot]\ =\ \sum_{n\ge 0}\IPplus^{(\kappa)}[S_{1}\in\cdot]^{*n} $$
is an ordinary renewal measure of a random walk with nonarithmetic increment distribution $\IPplus[S_{1}\in\cdot]$ and positive drift $\pppprime(\kappa)$. We also note that $g_{\kappa}*\IUplus^{(\kappa)}(t)=e^{\kappa t}\,g_{\kappa}*\IUplus(t)$. Hence, if $g_{\kappa}$ is dRi, then $g*\IUplus^{(\kappa)}(t)$ is bounded and converges to the limit stated in \eqref{eq:KRT case 2A+} by the key renewal theorem. Further, it then follows by the dominated convergence theorem that
$$ \Eminus\Big[e^{\kappa S_\tau}\,g*\IUplus^{(\kappa)}(t-S_\tau) 1_{\{\tau<\infty\}}\Big] \ \xrightarrow{t\to\infty}\ \frac{\Eminus e^{\kappa S_{\tau}}1_{\{\tau<\infty\}}}{\pppprime(\kappa)}\int_{\R}g_{\kappa}(x)\ dx $$
and thereby \eqref{eq:KRT case 2A-+} if
\begin{equation*}
\Eminus e^{\kappa S_{\tau}}\1_{\{\tau<\infty\}}\ =\ \frac{\pmp(\kappa)}{1-\pmm(\kappa)}.
\end{equation*}
But this follows from
\begin{align*}
\Eminus e^{\kappa S_{\tau}}\1_{\{\tau<\infty\}}\ &=\ \sum_{n\ge 1}\Eminus\big[e^{\kappa S_{n}}\1_{\{\tau=n\}}\big]\ =\,\sum_{n\ge 1}\IPminus^{(\kappa)}[\tau=n]\\
&=\ \sum_{n\ge 1}\pmm(\kappa)^{n-1}\ppm(\kappa)\ =\ \frac{\pmp(\kappa)}{1-\pmm(\kappa)}
\end{align*}
which completes the proof.
\qed \end{proof}

\vspace{.2cm}
\textit{\bfseries Case 2B}. $\ppm=0<\pmp$, \eqref{eq:kappa finite case 2} holds and $\pmm(\kappa)=1>\ppp(\kappa)$.\\[1mm] 
In this case $\kappa=\kappaminus$ and $(S_{n})_{n\ge 0}$ forms an ordinary random walk under each $\Prob_{\delta}^{(\kappa)}$ because $\Prob_{\delta}^{(\kappa)}[\xi_{n}=\delta$ for all $n\ge 0]=1$ (Lemma \ref{lem:Pdelta^kappa case 2BC}). With $\IUminus,\IUplus$ as defined before, the following result holds and is again shown by standard renewal-theoretic arguments. In order to state it, we need to define
$$ \IUminplus(D)\,:=\,\sum_{n\ge 1}\IPminus[\xi_{1}=1,\,S_{n}\in D]=\Eminus\big[\IUplus(D-S_{1})\1_{\{\xi_{1}=1\}}\big] $$
for measurable $D\subset\R$. For later use (see the proof of Thm.~\ref{thm:main case 2}), we also point out that
\begin{align}
\begin{split}\label{eq:mgf Uminplus at kappa}
\int_{\R}e^{\kappa x}\ \IUminplus(dx)\ &=\ \sum_{n\ge 1}\Eminus\big[\1_{\{\xi_{1}=1\}}e^{\kappa S_{n}}\big]\\
&=\ \Eminus\big[\1_{\{\xi_{1}=1\}}e^{\kappa S_{1}}\big]\sum_{n\ge 1}\Eplus e^{\kappa S_{n-1}}\\
&=\ \frac{\pmp(\kappa)}{1-\ppp(\kappa)}.
\end{split}
\end{align}
Hence, it is finite iff $\pmp(\kappa)<\infty$, as $\ppp(\kappa)<\pmm(\kappa)=1$.

\begin{Prop}\label{prop:RThm case 2B}
Under the stated assumptions, the following assertions hold true for any function $g:\R\to\R$ such that $g\in\cR_{\kappa}$ (with $\kappa=\kappaminus$): If $S_{1}$ is nonarithmetic under $\IPminus$, then $g*\IUminplus\in\cR_{\kappa}$ and
\begin{gather}
e^{\kappa t}\,g*\IUminus(t)\ \xrightarrow{t\to\infty}\ \frac{1}{\pmmprime(\kappa)}\int_{\R}g_{\kappa}(x)\ dx.\label{eq:KRT case 2B-}\\
\shortintertext{Moreover,}
e^{\theta t}\,g*\IUplus(t)\ =\ e^{\theta t}\,\Eplus\left[\sum_{n\ge 0}g(t-S_{n})\right]\ \xrightarrow{|t|\to\infty}\ 0,\label{eq:KRT case 2B+}
\end{gather}
for each $\theta>0$ such that $\ppp(\theta)<1$, $\pmp(\theta)<\infty$ and $g\in\cR_{\theta}$. If $\ppp(\theta)<1$, $\pmp(\theta)<\infty$ hold for some $\theta\in (\kappaminus,\kappaplus)$ and $g\in\cR_{\kappa}$, then $g*\IUminplus\in\cR_{\kappa}$ and
\begin{align}
\begin{split}\label{eq:KRT case 2B-+}
e^{\kappa t}\,\Eminus g*\IUplus(t-S_{\tau})\1_{\{\tau<\infty\}}\ =~&e^{\kappa t}\,\Eminus\left[\sum_{n\ge 0}g(t-S_{\tau+n})\1_{\{\tau<\infty\}}\right]\\
\xrightarrow{t\to\infty}~&\frac{1}{\pmmprime(\kappa)}\int_{\R}(g*\IUminplus)_{\kappa}(x)\ dx
\end{split}
\end{align}
Finally, if $\kappaplus<\infty$, $\IPplus[S_{1}\in\cdot|\xi_{1}=1]$ is nonarithmetic and $g\in\cR_{\kappaplus}$, then
\begin{equation}
e^{\kappaplus t}\,g*\IUplus(t)\ \xrightarrow{t\to\infty}\ \frac{1}{\pppprime(\kappaplus)}\int_{\R}g_{\kappaplus}(x)\ dx\label{eq2:KRT case 2B++}
\end{equation}
holds in addition to \eqref{eq:KRT case 2B+}.
\end{Prop}

\begin{proof}
In view of the proof of Prop.~\ref{prop:RThm case 2A}, only \eqref{eq:KRT case 2B-+} needs our attention.
Without loss of generality, let $g$ be nonnegative. Note also that $g$ is $\llam$-almost everywhere continuous ($\llam$ Lebesgue measure) because $g_{\kappa}$ is dRi. For each $t\in\R$, we have
\begin{align*}
e^{\kappa t}\,&\Eminus g*\IUplus(t-S_{\tau})\1_{\{\tau<\infty\}}\ =~e^{\kappa t}\,\Eminus\Bigg[\sum_{n\ge 0}g(t-S_{\tau+n})\1_{\{\tau<\infty\}}\Bigg]\\
&=\ e^{\kappa t}\sum_{k\ge 0}\Eminus\big[\1_{\{\xi_{k}=-1,\,\xi_{k+1}=1\}}\,g*\IUplus(t-S_{k}-X_{k+1})\big]\\
&=\ e^{\kappa t}\sum_{k\ge 0}\Eminus\big[\1_{\{\overline{\tau}>k\}}\,g*\IUminplus(t-S_{k})\big]\\
&=\ e^{\kappa t}\,\Eminus\Bigg[\sum_{k=0}^{\overline{\tau}-1}\,g*\IUminplus(t-S_{k})\Bigg]\ =\ (g*\IUminplus)_{\kappa}*\IUminus^{(\kappa)}(t).
\end{align*}
Hence, \eqref{eq:KRT case 2B-+}  follows by the key renewal theorem if we can verify that $(g*\IUminplus)_{\kappa}$ is dRi. To this end, pick $\eps>0$ so small that $\kappa+\eps\le\theta$, thus $\ppp(\kappa+\eps)<1$ and
$$ \pmp(\kappa + \eps)\ =\ \Eminus\big[\1_{\{\xi_{1}=1\}}e^{(\kappa + \eps)S_{1}}\big]\ <\ \infty. $$
Since $g*\IUminplus(t)=\Eminus\big[g*\IUplus(t-S_{1})\1_{\{\xi_{1}=1\}}\big]$
for $t\in\R$, we infer with the help of \eqref{eq:KRT case 2B+} that
\begin{align*}
(g*\IUminplus)_{\kappa}(t)\ &=\ e^{\kappa t}\,g*\IUminplus(t)\\
&\le\ C\,e^{\kappa t}\,\Eminus\big[e^{-(\kappa\pm\eps)(t-S_{1})}\1_{\{\xi_{1}=1\}}\big]\\
&\le\ C\,\big(\pmp(\kappa-\eps)\vee\pmp(\kappa+\eps)\big)\,e^{-\eps|t|}
\end{align*}
for some $C\in\IRg$, and this in combination with the $\llam$-almost everywhere continuity of $g$ mentioned above yields the desired result.
\qed \end{proof}

\section{Proof of Thm.~\ref{thm:main case 1}}\label{sec:proof main case 1}

For $\phi\in\cC^{*}(\R)$, let as usual $\|\phi\|_{\infty}$ be its supremum norm and $K_{\phi}$ the maximal positive value such that
\begin{equation}\label{eq:def Kphi}
|\phi(x)-\phi(y)|\ \le\ \Lip(\phi)\,|x-y|\,\1_{[K_{\phi},\infty)}(|x|\vee |y|)
\end{equation}
for all $x,y\in\IRg$. Note that in addition to \eqref{eq:def Kphi}, we have
\begin{equation}\label{eq:2nd estimate phi}
|\phi(x)-\phi(y)|\ \le\ 2\,\|\phi\|_{\infty}\,\1_{[K_{\phi},\infty)}(|x|\vee |y|)
\end{equation}
for all $x,y\in\IRg$. Given a stationary distribution $\nu$ of $(X_{n})_{n\ge 0}$, let $R$ be a generic random variable with this law independent of all other occurring random variables, notably $\Psi,\Lambda,\Aminus,\Aplus$ and $B$.

\vspace{.2cm}
The following two lemmata do not require the assumption $\pmp\wedge\ppm>0$ and will also be used in the proof of our main result in the unilateral case.

\begin{Lemma}\label{lem:moments of R}
Assuming $\rho(\theta)<1$ and $\Erw B^{\theta} <\infty$, the random variable $R$ with law $\nu$ satisfies $\Erw|R|^{\theta}<\infty$ for all $0<\theta<\kappa$.
\end{Lemma}

\begin{proof}
With $\wh{Y}_{n}$ as defined in \eqref{eq:hatY_{n}} for $n\in\N_{0}$ and
$$ \wh{Y}_{\infty}\,:=\,B_{1}\ +\ \sum_{n\ge 1}\Lip(\Lambda_{1}\cdots\Lambda_{n})B_{n+1}, $$
Lemma \ref{lem:comparison with LIFS} provides us with
\begin{align*}
\Prob[|R|>t]\ &=\ \Prob[|\Psi_{1}\cdots\Psi_{n}(R)|>t]\\
&\le\ \Prob[\wh{Y}_{n}+|\Lambda_{1}\cdots\Lambda_{n}(R)|>t]\ \xrightarrow{n\to\infty}\ \Prob[\wh{Y}_{\infty}>t]
\end{align*}
for all $t\ge 0$, where $\Lambda_{1}\cdots\Lambda_{n}(R)\to 0$ in probability has also been used which in turn holds because, by Lemma \ref{lem:Ptheta vs Lambda},
\begin{gather}\label{WLLN Lambda}
n^{-1}\log\|\Lip(\Lambda_{1}\cdots\Lambda_{n})\|_{\theta}\ \xrightarrow{n\to\infty}\ \log\rho(\theta)^{1/(\theta\vee 1)}\,<\,0
\end{gather}
for $\theta\in (0,\kappa)$. Hence, it suffices to show $\Erw\wh{Y}_{\infty}^{\theta}<\infty$ for $\theta\in (0,\kappa)$. Putting $\|X\|_{\theta}:=\Erw|X|^{\theta}$ for $0<\theta\le 1$ and $:=(\Erw|X|^{\theta})^{1/\theta}$ for $\theta\ge 1$, we find
\begin{gather}\label{eq:theta-norm whYinfty}
\|\wh{Y}_{\infty}\|_{\theta}\ \le\ \|B\|_{\theta}\left(1+\sum_{n\ge 1}\big(e^{\log\|\Lip(\Lambda_{1}\cdots\Lambda_{n})\|_{\theta}/n}\big)^{n}\right)\ <\ \infty
\end{gather}
where \eqref{WLLN Lambda} has once again been utilized.
\qed \end{proof}

The following lemma about the direct Riemann integrability of certain functions appearing in the proofs of our main results is crucial and formulated in such a way that it can be used in any of these proofs, there for $\kappa=\kappaminus\wedge\kappaplus$ as one should expect. We also note that the moment condition on $R$ is guaranteed by Lemma \ref{lem:moments of R}.

\begin{Lemma}\label{lem:dRi case 1}
Let $R$ be as stated before Lemma \ref{lem:moments of R} and $\kappa>0$ such that
\begin{align}\label{eq:dRi moment assumptions}
\Erw|\Apm|^{\kappa}\,<\,\infty,\quad\Erw|B|^{\kappa}<\infty\quad\text{and}\quad\Erw|R|^{\theta}\,<\,\infty\text{ for }\theta\in(0,\kappa).
\end{align}
Further defining
\begin{gather*}
h_{\phi}\suppm(x)\,:=\,\Erw\big[\phi(\delta e^{-x}|\Psi(R)|)\1_{\{\pm\Psi(R)>0\}}-\phi(\delta e^{-x}|\Lambda(R)|)\1_{\{\pm\Lambda(R)>0\}}\big]
\end{gather*}
for $x\in\R$ and any $\phi\in\cC_\delta^{*}(\R)$, $\delta\in\bbS^{0}$, the function $\wh{h}_{\phi}^{(\epsilon)}(x):=e^{\kappa x}h_{\phi}^{(\epsilon)}(x)$ is dRi, i.e.~$h_{\phi}^{(\epsilon)}\in\cR_{\kappa}$ for each $\epsilon\in\bbS^{0}$.
Furthermore the function
$$ \overline{h}_{\varphi}(x)\,:=\,\Erw\left[|\varphi(e^{-x}\Psi(R))-\varphi(e^{-x}\Lambda(R))|\right] $$
is also in $\cR_{\kappa}$ for any $\varphi\in\cC^{*}(\R)$.
\end{Lemma}

\begin{proof} Define $\varphi(t)=\phi(\delta|t|)\1_{\{\epsilon t>0\}}$ for $\epsilon=\pm 1$ and
$$ h_{\varphi}(x)\,:=\,\Erw\left[\varphi(e^{-x}\Psi(R))-\varphi(e^{-x}\Lambda(R))\right]. $$
Then $h_{\varphi}=h_{\phi}^{\epsilon}$ because $\Psi(R)$ and $e^{-x}\Psi(R)$ (resp.~$\Lambda(R)$ and $e^{-x}\Lambda(R)$) have the same sign, and it suffices to show that, for any Lipschitz function $\varphi$ in $\cC^{*}(\R)$,
\begin{equation}\label{eq: sup}
\sum_{n\in\Z}\sup_{n<x\le n+1}e^{\kappa x}|\overline{h}_{\varphi}(x)|\ <\ \infty.
\end{equation}
Furthermore, the range of summation may be reduced to $n\in\N_{0}$ because
$$ \sum_{n\ge 1}\sup_{-n<x\le-n+1}e^{\kappa x}|\overline{h}_{\varphi}(x)|\ \le\ 2
\,\|\varphi\|_{\infty}\sum_{n\ge 0}e^{-\kappa n}\ <\ \infty. $$

\vspace{.1cm}
Put $M:=|\Psi(R)|\vee|\Lambda(R)|$ and observe that $M\le|\Lambda(R)|+B$ by \eqref{eq:def2 AL}, and
\begin{align}
\overline{h}_{\varphi}(x)\ &= \Erw\left[ |\varphi(e^{-x}\Psi(R))-\varphi(e^{-x}\Lambda(R))|\right]\nonumber\\
&\le\ \,\Erw\left[(\Lip(\varphi)\,e^{-x}|\Psi(R)-\Lambda(R)|\,\1_{[K_{\varphi},\infty)}(e^{-x}M))\wedge (2\|\varphi\|_{\infty})\right]\nonumber\\
&\le\ \Erw\left[(\Lip(\varphi)B\,e^{-x})\wedge (2\|\varphi\|_{\infty})\,\1_{[K_{\varphi},\infty)}(e^{-x}M)\right]\label{eq:dRi phi init}
\end{align}
by \eqref{eq:def2 AL}, \eqref{eq:def Kphi} and \eqref{eq:2nd estimate phi}. Since
\begin{align*}
\Erw\Bigg[\sum_{0\le n\le\log B}\sup_{n<x\le n+1}&e^{\kappa x}|\varphi(e^{-x}\Psi(R))-\varphi(e^{-x}\Lambda(R))|\Bigg]\\
&\le\ 2\|\varphi\|_{\infty}\,e^{\kappa}\,\Erw\Bigg[B^{\kappa}\sum_{0\le n\le\log B}e^{\kappa(n-\log B)}\Bigg]\\
&\le\ 2\|\varphi\|_{\infty}\,e^{\kappa}(1-e^{-\kappa})^{-1}\,\Erw B^{\kappa}\ <\ \infty,
\end{align*}
it remains to show that
\begin{align*}
J\,:=\,\Erw\Bigg[\sum_{n>\log B}\sup_{n<x\le n+1}&e^{\kappa x}\big|\varphi(e^{-x}|\Psi(R)|)-\varphi(e^{-x}|\Lambda(R)|)\big|\Bigg]\ <\ \infty
\end{align*}
for the desired conclusion \eqref{eq: sup}. To this end, we further estimate
\begin{align}
J\ &\le\ \Lip(\varphi)\,\Erw\Bigg[\sum_{n>\log B}e^{(\kappa-1)n}|\Psi(R)-\Lambda(R)|\,\1_{[K_{\varphi},\infty)}(e^{-n}M)\Bigg]\nonumber\\
&\le\ \Lip(\varphi)\,\Erw\Bigg[B\sum_{\log B<n\le\log(M/K_{\varphi})}e^{(\kappa-1)n}\Bigg]\label{eq:bound for J}\\
&\le\
\begin{cases}
\hfill\displaystyle\frac{\Lip(\varphi)}{1-e^{\kappa-1}}\,\Erw B^{\kappa},&\text{if }\kappa<1\\[1mm]
\hfill \Lip(\varphi)\,\Erw B\log(M/K_{\varphi}B),&\text{if }\kappa=1\\[1mm]
\displaystyle\frac{\Lip(\varphi)\,K_{\varphi}^{1-\kappa}}{e^{\kappa-1}-1}\,\Erw B(M\vee 1)^{\kappa-1},&\text{if }\kappa>1
\end{cases}.\nonumber
\end{align}

If $\kappa<1$, then $\Erw B^{\kappa}<\infty$ is guaranteed by \eqref{eq:dRi moment assumptions}. For the case $\kappa=1$, we point out that
\begin{align*}
\Erw B\log(M/B)\ &\le\ \Erw B\log(1+\Lambda(R)/B)\ \le\ \Erw B\log(1+\Lip(\Lambda)|R|/B)\\
&\le\ \Erw B\log(1+\Lip(\Lambda)/B)\,+\,\Erw\log(1+|R|)
\end{align*}
where $\Erw\log(1+|R|)<\infty$ by another appeal to \eqref{eq:dRi moment assumptions}. Left with the first expectation in the previous display, the inequality
$$ x\log\left(1+\frac{y}{x}\right)\ =\ y\left[\frac{x}{y}\,\log\left(1+\frac{y}{x}\right)\right]\ \le\ y $$
for $0<x<y$ combined with $\Lip(\Lambda)=|\Aminus|\vee|\Aplus|$ and \eqref{eq:dRi moment assumptions} provides us with
\begin{align*}
\Erw B\log(1+\Lip(\Lambda)/B)\ &\le\  (\log 2)\,\Erw B\,+\,\Erw B\log(1+\Lip(\Lambda)/B)\1_{\{B<\Lip(\Lambda)\}}\\
&\le\ 2\,\Erw B\,+\,\Erw\Lip(\Lambda)\ <\ \infty.
\end{align*}

Finally, if $\kappa>1$, note first that $\Erw B(M\vee 1)^{\kappa-1}$ is bounded by a constant times $\Erw B\Lip(\Lambda)^{\kappa-1}\,\Erw|R|^{\kappa-1}+\Erw B^{\kappa}$. Use H\"older's inequality to infer
\begin{align*}
\Erw B\Lip(\Lambda)^{\kappa-1}\ \le\ (\Erw B^{\kappa})^{1/\kappa}\,(\Erw\Lip(\Lambda)^{\kappa})^{(\kappa-1)/\kappa}.
\end{align*}
Now $\Erw B(M\vee 1)^{\kappa-1}<\infty$ follows again by \eqref{eq:dRi moment assumptions}.
\qed \end{proof}

\textsc{Proof} \textit{(of Thm.~\ref{thm:main case 1})}
As \eqref{eq:tails main case 1} is an almost immediate consequence of \eqref{eq:KRT main case 1}, it suffices to prove the latter and identity \eqref{eq:relation Cplusminus}.

\vspace{.1cm}
Let $R$ be as in the previous lemma and independent of the i.i.d.~random variables $(\Psi,\Lambda),(\Psi_{1},\Lambda_{1}),(\Psi_{2},\Lambda_{2}),\ldots$ Observe that $(\xi_{n},S_{n})_{n\ge 0}$ and $\Psi(R),\Lambda(R)$ are  independent, thus
\begin{align*}
&\Prob\left[ (\sign(\Lambda_{n}\cdots\Lambda_{1}(x)),\log|\Lambda_{n}\cdots\Lambda_{1}(x)|)_{n\ge 0}\in\cdot\ \right]\\
&=\ \Prob_{\sign(x)}[(\xi_{n},S_{n}+\log|x|)_{n\ge 0}\in\cdot\ ]\\
&=\Prob\left[ (\xi_{n},S_{n}+\log|x|)_{n\ge 0}\in\cdot\ |\xi_{0}=\sign(x),\Psi(R),\Lambda(R)\right]\
\end{align*}
for each $x\in\R$. Defining
$$ \phi\star\nu(t)\,:=\,\int\phi(e^{-t}x)\,\nu(dx)\quad \text{and}\quad\phi_{\delta}(t)\,:\,=\phi(t)\1_{\IRg}(\delta t)\,\in\,\cC_{\delta}^{*}(\R) $$
for $\delta\in\bbS^{0}$, we have
\begin{gather}
\phi\star\nu(t)\ =\ \phi_{-1}\star\nu(t)+\phi_{1}\star\nu(t)\label{eq:potential formula nu(phi)}
\intertext{and will prove that $\rho(\theta)<1$ for $\theta\in (0,\kappa)$ and $\Erw B^\kappa <\infty$ are enough to infer}
\phi_{\delta}\star\nu(t)\,=\,\sum_{\epsilon\in\bbS^{0}}\sum_{n\ge 0}\Erw_{\epsilon} h_{\phi_{\delta}}^{(\epsilon)}(t-S_{n})\1_{\{\xi_{n}=\delta\}}\label{eq2:potential formula nu(phi)}
\end{gather}
for all $t\in\R$ and $\delta\in\bbS^{0}$. In particular, irreducibility $(\ppm\wedge \pmp>0)$ is not required, a fact we will take advantage of later when dealing with the other cases.
Note that, by Lemmata \ref{lem:Ptheta vs Lambda} and \ref{lem:moments of R},
$$ \left(\Erw|\Lambda_{n}\cdots\Lambda_{1}(R)|^{\theta}\right)^{1/n}\ \le\ \left(\Erw \Lip(\Lambda_{n}\cdots\Lambda_{1})^{\theta}\Erw|R|^{\theta}\right)^{1/n}\ \xrightarrow{n\to\infty}\ \rho(\theta)\ <\ 1$$
This entails that, for any $\phi\in\cC^*(\R)$ and with $C$ such that $\phi(x)\le C|x|^{\theta}$,
$$ \Erw\phi(e^{-t}\Lambda_{n}\cdots\Lambda_{1}(R))\ \le\ Ce^{-\theta t}  \Erw|\Lambda_{n}\cdots\Lambda_{1}(R)|^{\theta}\ \xrightarrow{n\to\infty}\ 0, $$
and then
\begin{align}
\phi_{\delta}&\star\nu(t)\ =\ \sum_{j=0}^{n-1}\left[\Erw\phi_{\delta}(e^{-t}\Lambda_{j}\cdots\Lambda_{1}(R))\,-\,\Erw\phi_{\delta}(e^{-t}\Lambda_{j+1}\cdots\Lambda_{1}(R))\right]\nonumber\\
&\hspace{1.7cm}+\ \Erw\phi_{\delta}(e^{-t}\Lambda_{n}\cdots\Lambda_{1}(R))\nonumber\\
&=\ \sum_{j=0}^{n-1}\Erw\!\left[\phi_{\delta}(e^{-t}\Lambda_{j}\cdots\Lambda_{1}(\Psi(R)))-
\phi_{\delta}(e^{-t}\Lambda_{j}\cdots\Lambda_{1}(\Lambda(R)))\right]\,+\,o(1)
\label{eq:tail equal potential}
\end{align}
as $n\to\infty$. Moreover,
\begin{align*}
&\Erw\!\left[\phi_{\delta}(e^{-t}\Lambda_{j}\cdots\Lambda_{1}(x))\right]\ =\ \Erw_{\sign(x)}\Big[\phi_{\delta}(\xi_{j}e^{-(t-S_{j})}|x|)\Big]\\
&=\ \Erw_{\sign(x)}\Big[\phi_{\delta}(\delta e^{-(t-S_{j})}|x|)\1_{\{\xi_{j}=\delta\}}\Big]\\
&=\ \Eplus\Big[\phi_{\delta}(\delta e^{-(t-S_{j})}|x|)\1_{\{x>0\}}\1_{\{\xi_{j}=\delta\}}\Big]+\Eminus\Big[\phi_{\delta}(\delta e^{-(t-S_{j})}|x|)\1_{\{x<0\}}\1_{\{\xi_{j}=\delta\}}\Big]
\end{align*}
for all $x\in \R$, hence
\begin{align*}
&\Erw\Big[\phi_{\delta}(e^{-t}\Lambda_{j}\cdots\Lambda_{1}(\Psi(R)))-
\phi_{\delta}(e^{-t}\Lambda_{j}\cdots\Lambda_{1}(\Lambda(R)))\Big]\\
&=\ \Eplus\Big[\Big(\phi_{\delta}(\delta e^{-(t-S_{j})}|\Psi(R)|)\1_{\{\Psi(R)>0\}}- \phi_{\delta}(\delta e^{-(t-S_{j})}|\Lambda(R)|)\1_{\{\Lambda(R)>0\}}\Big)\1_{\{\xi_{j}=\delta\}}\Big]\\
&+\ \Eminus\Big[\Big(\phi_{\delta}(\delta e^{-(t-S_{j})}|\Psi(R)|)\1_{\{\Psi(R)<0\}}-\phi_{\delta}(\delta e^{-(t-S_{j})}|\Lambda(R)|)\1_{\{\Lambda(R)<0\}}\Big)\1_{\{\xi_{j}=\delta\}}\Big]\\
&=\ \Eplus h_{\phi_{\delta}}^{(+)}(t-S_{j})\1_{\{\xi_{j}=\delta\}}\ +\ \Eminus h_{\phi_{\delta}}^{(-)}(t-S_{j})\1_{\{\xi_{j}=\delta\}}
\end{align*}
for each $\delta\in\bbS^{0}$ and $j\in\N$. By combining this with \eqref{eq:tail equal potential}, we obtain \eqref{eq2:potential formula nu(phi)}.

\vspace{.1cm}
Since, by Lemma \ref{lem:dRi case 1}, the $\wh{h}_{\phi_{\delta}}^{(\epsilon)}(t):=e^{\kappa t}h_{\phi_{\delta}}^{(\epsilon)}(t)$ are dRi for $\delta,\epsilon\in\bbS^{0}$, we infer with the help of Prop.~\ref{prop:RThm case 1} that
$$ e^{\kappa t}\,\phi\star\nu(t)\ \xrightarrow{t\to\infty}\ \Theta(\phi_{1})+\Theta(\phi_{-1}), $$ where
\begin{gather*}
\Theta(\phi_{\delta})\,:=\,\frac{u_{\delta}(\kappa)\sum_{\epsilon\in\bbS^{0}}v_{\epsilon}(\kappa)\int_{-\infty}^{\infty}\wh{h}_{\phi_{\delta}}^{(\epsilon)}(x)\ dx}{\whpim(\kappa)\,\Erw|\Aminus|^{\kappa}\log|\Aminus|\,+\,\whpip(\kappa)\,\Erw|\Aplus|^{\kappa}\log|\Aplus|}
\end{gather*}
where, using two substitutions and Fubini's Theorem (absolute integrability is guaranteed by Lemma \ref{lem:dRi case 1}),
\begin{align*}
&\int_{-\infty}^{\infty}\wh{h}_{\phi_{\delta}}^{(\epsilon)}(x)\ dx\\
&=\ \int_{-\infty}^{\infty} e^{\kappa x}\,\Erw\big[\phi_{\delta}(\delta e^{-x}|\Psi(R)|)\1_{\{\epsilon\Psi(R)>0\}}-\phi_{\delta}(\delta e^{-x}|\Lambda(R)|\1_{\{\epsilon\Lambda(R)>0\}})\big]\,dx\\
&=\ \int_{0}^{\infty} \frac{1}{x^{\kappa+1}}\Erw\big[\phi_{\delta}(\delta |\Psi(R)|x)\1_{\{\epsilon\Psi(R)>0\}}-\phi_{\delta}(\delta |\Lambda(R)|x)\1_{\epsilon\Lambda(R)>0\}}\big]\,dx\\
&=\ \Erw\Big[|\Psi(R)|^{\kappa}\1_{\{\epsilon\Psi(R)>0\}}-|\Lambda(R)|^{\kappa}\1_{\{\epsilon\Lambda(R)>0\}}\Big]\int_{0}^{\infty}\frac{\phi_\delta(\delta x)}{x^{\kappa+1}}\,dx.
\end{align*}
This yields $\Theta(\phi_{1})=\Cplus\int_{0}^{\infty}\frac{\phi(x)}{x^{\kappa+1}}\,dx$ with
\begin{equation}\label{eq:def Cplus case 1}
\Cplus\ =\ \frac{\uplus(\kappa)\sum_{\epsilon}v_{\epsilon}(\kappa)\,\Erw\big[|\Psi(R)|^{\kappa}\1_{\{\epsilon\Psi(R)>0\}}-|\Lambda(R)|^{\kappa}\1_{\{\epsilon\Lambda(R)>0\}}\big]}{\whpim(\kappa)\,\Erw|\Aminus|^{\kappa}\log|\Aminus|\,+\,\whpip(\kappa)\,\Erw|\Aplus|^{\kappa}\log|\Aplus|}
\end{equation}
and accordingly  $\Theta(\phi_{-1})=\Cminus\int_{0}^{\infty}\frac{\phi(-x)}{x^{\kappa+1}}\,dx$ with
\begin{equation}\label{eq:def Cminus case 1}
\Cminus\ =\ \frac{\uminus(\kappa)\sum_{\epsilon}v_{\epsilon}(\kappa)\,\Erw\big[|\Psi(R)|^{\kappa}\1_{\{\epsilon\Psi(R)>0\}}-|\Lambda(R)|^{\kappa}\1_{\{\epsilon\Lambda(R)>0\}}\big]}{\whpim(\kappa)\,\Erw|\Aminus|^{\kappa}\log|\Aminus|\,+\,\whpip(\kappa)\,\Erw|\Aplus|^{\kappa}\log|\Aplus|}.
\end{equation}
This completes the proof of \eqref{eq:KRT main case 1}, and Relation \eqref{eq:relation Cplusminus} is now a direct consequence of the two formulae for $\Cminus$ and $\Cplus$.
\hfill $\square$

\section{Proof of Thm.~\ref{thm:main case 2}, Parts (b) and (c)}\label{sec:proof main case 2}

Let $\phi\in\cC^{*}(\R)$. Being in the unilateral case, $\ppm=0<\pmp$ entails that $\IPplus[\xi_{n}=-1]=0$ for all $n\ge 0$ and therefore that decomposition
\eqref{eq:potential formula nu(phi)} of $\phi*\nu$ simplifies to
\begin{align}\label{eq:unilateral decomp of phi*nu}
\phi\star\nu(t)\ =\ h_{\phi_{-1}}\supminus*\IUminus(t)\ +\ \Eminus\Bigg[\sum_{n\ge\tau}h_{\phi_{1}}\supminus(t-S_{n})\Bigg]\ +\ h_{\phi_{1}}\supplus*\IUplus(t)
\end{align}
with $\tau=\tau(1)$ and $\IUminus,\,\IUplus$ as defined before Prop.~\ref{prop:RThm case 2A}. In particular,
\begin{align}\label{eq:unilateral decomp of phi*nu simplified}
\phi\star\nu(t)\ =\ h_{\phi}\supminus*\IUminus(t)
\end{align}
if $\phi\in\cC_{-}^{*}(\R)$ and thus $\phi=\phi_{-1}$. In order to prove Thm.~\ref{thm:main case 2}, we will use the above decomposition \eqref{eq:unilateral decomp of phi*nu} and determine asymptotics for its terms on the right-hand side with the help of Props. \ref{prop:RThm case 2A} and \ref{prop:RThm case 2B} in combination with Lemmata \ref{lem:moments of R} and \ref{lem:dRi case 1} which ensure the direct Riemann integrability of the functions $\wh{h}_{\phi_{\delta}}^{(\epsilon)}(t)=e^{\kappa t}h_{\phi_{\delta}}^{(\epsilon)}(t)$ for $\delta,\epsilon\in\bbS^{0}$. As usual, let $\kappa=\kappaminus\wedge\kappaplus$. In particular $\rho(\kappa)=1$ and $\rho(\theta)<1$ for any $\theta\in [0,\kappa)$.
\vspace{.2cm}

(b) Here $\kappa=\kappaplus$. Given any $\phi\in\cC^{*}(\R)$, we use \eqref{eq:unilateral decomp of phi*nu}. By Prop.~\ref{prop:RThm case 2A}, the first term on the right-hand side is of the order $o(e^{-\kappaplus t})$ as $t\to\infty$, the second term convergent to $\Cplus^{(1)}\int_{0}^{\infty}x^{-(\kappaplus+1)}\,\phi(x)\,dx$ with
\begin{align}\label{Cplus1 Thm 2.2(b)}
\Cplus^{(1)}\,:=\,\frac{\pmp(\kappaplus)\,\Erw\big[|\Psi(R)|^{\kappaplus}\1_{\{\Psi(R)<0\}}-|\Lambda(R)|^{\kappaplus}\1_{\{\Lambda(R)<0\}}\big]}{(1-\pmm(\kappaplus))\ppp'(\kappaplus)},
\end{align}
and the last one convergent to $\Cplus^{(2)}\int_{0}^{\infty}x^{-(\kappaplus+1)}\,\phi(x)\,dx$ with
\begin{align}\label{Cplus2 Thm 2.2(b)}
\Cplus^{(2)}\,:=\,\frac{\Erw\big[|\Psi(R)|^{\kappaplus}\1_{\{\Psi(R)>0\}}-|\Lambda(R)|^{\kappaplus}\1_{\{\Lambda(R)>0\}}\big]}{\ppp'(\kappaplus)}.
\end{align}
Consequently, \eqref{eq:KRT2 main case 2} holds with
\begin{gather}
\Cplus\ =\ \Cplus^{(1)}\,+\,\Cplus^{(2)}.\label{Cplus Thm 2.2(b)}
\end{gather}

\vspace{.2cm}
(c) Again, pick any $\phi\in\cC^{*}(\R)$ and note that $\kappa=\kappaminus$. In this case, $e^{\kappa t}$ times the first term on the left-hand side in \eqref{eq:unilateral decomp of phi*nu} converges to $\Cminus\int_{0}^{\infty}\phi(-x)\,dx$ with $\Cminus$ given by
\begin{equation}\label{Cminus Thm 2.2(a)}
\Cminus\ =\ \frac{\Erw\big[|\Psi(R)|^{\kappaminus}\1_{\{\Psi(R)<0\}}-|\Lambda(R)|^{\kappaminus}\1_{\{\Lambda(R)<0\}}\big]}{\pmm'(\kappaminus)}.
\end{equation}
By \eqref{eq:KRT case 2B+} of Prop.~\ref{prop:RThm case 2B}, $e^{\kappa t}$ times the last term in \eqref{eq:unilateral decomp of phi*nu} converges to $0$. Left with an inspection of the middle term multiplied by $e^{\kappa t}$, use \eqref{eq:KRT case 2B-+} of Prop.~\ref{prop:RThm case 2B} to infer
\begin{align*}
e^{\kappa t}\,\Eminus\Bigg[\sum_{n\ge\tau}h_{\phi_{1}}\supminus*\IUplus(t-S_{n})\Bigg]\ \xrightarrow{t\to\infty}\ \frac{1}{\pmmprime(\kappa)}\int_{\R}(h_{\phi_{1}}\supminus*\IUminplus)_{\kappa}(x)\ dx.
\end{align*}
By proceeding in a similar manner as for the derivation of \eqref{eq:def Cplus case 1} and \eqref{eq:def Cminus case 1} (using partial integration and substitution), we find that
\begin{gather*}
\int_{\R}(h_{\phi_{1}}\supminus*\IUminplus)_{\kappa}(x)\ dx\ =\ \int_{\R}e^{\kappa x}\,h_{\phi_{1}}\supminus*\IUminplus(x)\ dx
\end{gather*}
and this is readily evaluated as $\int_{0}^{\infty}\frac{\phi(x)}{x^{\kappa+1}}\,dx$ times
$$ \Erw\big[|\Psi(R)|^{\kappa}\1_{\{\Psi(R)<0\}}-|\Lambda(R)|^{\kappa}\1_{\{\Lambda(R)<0\}} \cdot\int_{\R}e^{\kappa x}\,\IUminplus(x)\,dx. $$
Recalling \eqref{eq:mgf Uminplus at kappa} for the last term in the previous line, this shows that
\begin{gather}
e^{\kappa t}\,\Eminus\Bigg[\sum_{n\ge\tau}h_{\phi_{1}}\supminus*\IUplus(t-S_{n})\Bigg]\ \xrightarrow{t\to\infty}\ \Cminplus\int_{0}^{\infty}\frac{\phi(x)}{x^{\kappa+1}}\ dx\nonumber
\shortintertext{with}
\Cminplus\ =\ \frac{\pmp(\kappa)\,\Erw\big[|\Psi(R)|^{\kappa}\1_{\{\Psi(R)<0\}}-|\Lambda(R)|^{\kappa}\1_{\{\Lambda(R)<0\}}\big]}{\pmmprime(\kappa)(1-\ppp(\kappa))}.\label{Cminplus Thm 2.2(c)}
\end{gather}
A combination of the previous results yields \eqref{eq:KRT3 main case 2}.
\hfill $\square$

\section{Proofs of Thm.~\ref{thm:main case 2}(a) and Thm.~\ref{thm:main case 3}}\label{sec:proof main case 3}

Part (a) of Thm.~\ref{thm:main case 2} and Thm.~\ref{thm:main case 3} can both be deduced from Goldie's implicit renewal theory (see \cite[Thm. 2.3 and Cor 2.4]{Goldie:91} and also \cite{Mirek:11} and \cite{Alsmeyer:16}). We confine ourselves to details regarding Thm.~\ref{thm:main case 2}(a) because those for Thm.~\ref{thm:main case 3} are similar.

\begin{proof}[\it Proof of Thm.~\ref{thm:main case 2}(a)]
The claimed left tail behavior of $\nu$ at $-\infty$ follows directly with the help of Cor.~2.4 in \cite{Goldie:91} after checking the following conditions: With $\wtil{A}:=\Aminus\vee 0$,
\begin{gather*}
\Erw{\wtil{A}}^{\kappaminus}=1,\quad\Erw{\wtil{A}}^{\kappaminus}\log\wtil{A}<\infty,\quad
\text{the law of $\wtil{A}$ is nonarithmetic},
\shortintertext{and}
\Erw\left|\big|\Psi(R)\wedge 0\big|^{\kappaminus}-\big|\wtil{A}R\wedge 0\big|^{\kappaminus}\right|\,<\,\infty.
\end{gather*}
But the first three of them are immediate by the assumptions of Thm.~\ref{thm:main case 2}, and the last condition follows from the observation that
\begin{gather*}
\sup_{x\in\R}|\Psi(x)\wedge 0-\wtil{A}x\wedge 0|\ \le\ B\quad\text{a.s.}
\end{gather*}
and $\Erw B^{\kappaminus}<\infty$ (see \eqref{eq:moments AlogAminus B}).
\qed \end{proof}

As already said, the proof of Thm.~\ref{thm:main case 3} follows along the same lines, for part (b) using a conjugation with a homeomorphism $r:\R\to [1,+\infty)$.

\section{Positivity of the constants} \label{sec:positivity}

The purpose of this section is to provide conditions that entail positivity of the constants $\Cplus,\Cminus$ figuring in Thms.~\ref{thm:main case 1}, \ref{thm:main case 2} and \ref{thm:main case 3}. This does usually not follow from the existence of the limit and therefore requires additional arguments. Our approach here is based on a recent paper \cite{buraczewski:damek} and consists in proving that, if the support of the stationary measure $\nu$ is unbounded and some contraction property holds, then the constants are indeed positive. The results are stated in Props.~\ref{prop:tail constants positive}, \ref{prop:tail constants positive2} and \ref{prop:tail constants positive3} below, but we 
confine ourselves to the proof in the irreducible case because the remaining ones can be either  treated in an analogous way or reduced to Goldie's implicit renewal theory \cite{Goldie:91}.

\medskip

\textit{\bfseries Case 1 (irreducible case)}. $p_{-+}>0$ and $p_{+-}>0$.

\begin{Prop}\label{prop:tail constants positive}
In the situation of Thm.~\ref{thm:main case 1}, the constants $\Cplus$ and $\Cminus$ in \eqref{eq:tails main case 1} are strictly positive if the stationary measure $\nu$ has unbounded support.
\end{Prop}

The matrix $P(0)=P$ has dominant eigenvalue $\rho(0)=1$ because it is a transition matrix. It follows that the function $s\mapsto \rho(s)$ is smooth, convex (Lemma \ref{lem:properties rho(theta) case 1}) and satisfies $\rho(0)=\rho(\kappa)=1$ under our hypotheses which in combination with Lemma \ref{lem:rho(s)<1} in the Appendix further entails $\rho(s)<1$ for each $s\in (0,\kappa)$, a fact that will be used for the proof of the proposition (see after \eqref{eq:star}).

\vspace{.1cm}
Recall from Section \ref{sec:associated LIFS} that $\Lip(\Lambda_{n})=|\Aminus_{n}|\vee|\Aplus_{n}|$, $\wh{Z}_{n}=\Lambda_{1}\cdots\Lambda_{n}(X_{0})$., and let
\begin{equation*}
\wh{Y}_{\infty}= \sum_{n\ge 0}\Lri_{n}B_{n+1}
\end{equation*}
with $\Lri_{n}$ as defined after \eqref{eq:hatY_{n}}. Finally, let $X_{0}=\xi_{0}$ throughout this section and
$$ T_{t}\,:=\,\inf\{n\ge 1:\wh{Z}_{n}>t\} $$
for $t>0$. The proof of Prop.~\ref{prop:tail constants positive} is based on the subsequent lemma which does not require irreducibility.

\begin{Lemma}\label{lem:lower bound t^kappa nu}
Let $\nu$ be a stationary distribution with support unbounded to the right. If
\begin{gather}
\IPplus\big[T_{t}<\infty\big]\ \ge\ K_{1}t^{-\kappa}\label{eq1:lower bound t^kappa nu}
\shortintertext{and}
\Prob\big[\wh{Y}_{\infty}>t\big]\ \le\ K_{2}\,t^{-\kappa},\label{eq2:lower bound t^kappa nu}
\end{gather}
for suitable constants $K_{1},K_{2}>0$ and all $t\ge 1$, then
$$ \liminf_{t\to\infty}t^{\kappa}\,\nu((t,\infty))\,>\,0. $$
\end{Lemma}

\begin{proof}
We first show that
\begin{equation}\label{eq:lower bound nu}
\nu((t,\infty))\ \ge\ \Big(\Prob[\wh{Z}_{T_t}>t]-\Prob[\wh{Y}_{\infty}>(K-1)t]\Big)\,\nu((K,\infty))
\end{equation}
for all $t,K>0$. Fix $t$, write $T$ as shorthand for $T_{t}$, let $\cG_{n}$ denote the $\sigma$-field generated by $\Psi_{1},\ldots,\Psi_{n}$ for $n\ge 1$, and put $\Psi_{1}\cdots\Psi_{T}(x)=\Lambda_{1}\cdots\Lambda_{T}(x):=0$ for any $x\in\R$ if $T=\infty$. Observe that, by the $\Psi$-invariance of $\nu$ and the independence of $(\Psi_{n})_{n\ge 1}$ and $R$, the sequence
$$ M_{n}\,:=\,\Prob[\Psi_{1}\cdots\Psi_{n}(R)>t|\cG_{n}]\ =\ \int\1_{(t,\infty)}(\Psi_{1}\cdots\Psi_{n}(x))\ \nu(dx),\quad n\ge 1 $$
forms a bounded martingale under $\IPplus$ with mean $\nu(t)$. Therefore the optional sampling theorem provides  
$$ \nu((t,\infty))\ =\ \Eplus\left[\int\1_{\{\Psi_{1}\cdots\Psi_{T\wedge n}(x)>t\}}\ \nu(dx)\right] $$
for each $n\in\N$. Now use \eqref{eq:AL backward} of Lemma \ref{lem:comparison with LIFS} and
$$ \Lambda_{1}\cdots\Lambda_{T}(x)\ =\ x\Lambda_{1}\cdots\Lambda_{T}(1)\ =\ x\wh{Z}_{T}\quad\IPplus\text{-a.s.~for all }x>0 $$
to obtain upon passing to the limit $n\to\infty$ that
\begin{align*}
\nu((t,\infty))\ &=\ \Eplus\left[\int\1_{\{\Psi_{1}\cdots\Psi_{T}(x)>t,\ 
T<\infty  \}}\ \nu(dx)\right]\\
&\ge\ \Eplus\left[\int_{[K,\infty)}\1_{\{\Lambda_{1}\cdots\Lambda_{T}(x)-\wh{Y}_{T}>t,\ T<\infty\}}\ \nu(dx)\right]\\
&\ge\ \IPplus\big[T<\infty,\,K\wh{Z}_{T}-\wh{Y}_{T}>t\big]\,\nu([K,\infty))\\
&\ge\ \IPplus\big[T<\infty,\,\wh{Y}_{T}\le (K-1)t\big]\,\nu([K,\infty))\\
&=\ \Big(\IPplus\big[T<\infty\big]-\Prob\big[\wh{Y}_{T}>(K-1)t\big]\Big)\,\nu([K,\infty))\\
&\ge\ \Big(\IPplus\big[T<\infty\big]-\Prob\big[\wh{Y}_{\infty}>(K-1)t\big]\Big)\,\nu([K,\infty))
\end{align*} 
for any $K>0$ and thus \eqref{eq:lower bound nu}. By finally combining this for sufficiently large $K$ with \eqref{eq1:lower bound t^kappa nu}, \eqref{eq2:lower bound t^kappa nu}, we conclude the assertion of the lemma.
\qed \end{proof}

In view of this lemma, the proof of Prop.~\ref{prop:tail constants positive} reduces to a verification of \eqref{eq1:lower bound t^kappa nu} and \eqref{eq2:lower bound t^kappa nu}. The first of these conditions is shown as part of the next lemma, the second one in Lemma \ref{lem:aux2 pos const}.

\begin{Lemma}\label{lem:aux1 pos const}
Under the hypotheses of Prop.~\ref{prop:tail constants positive}, there exists a constant $K>0$ such that, for all $t\ge 1$,
\begin{gather}
t^{\kappa}\,\Prob\left[\sup_{n\ge 1}\Lri_{n}>t\right]\ \ge\ K^{-1},\label{eq1:aux1 pos const}
\shortintertext{and}
\sup_{n\ge 1}t^{\kappa}\,\Prob\left[\Lri_{n}>t\right]\ \le\ K.\label{eq3:aux1 pos const}
\end{gather}
Moreover, \eqref{eq1:lower bound t^kappa nu} holds, in fact
\begin{align}\label{eq2:aux1 pos const}
t^{\kappa}\,\Prob_{\delta}[T_{t}<\infty]\ =\ t^{\kappa}\,\Prob\left[\sup_{n\ge 1}\Lambda_{1}\cdots\Lambda_{n}(\delta)>t\right]\ \ge\ K
\end{align}
for some $K>0$ and each $\delta\in\bbS^{0}$.
\end{Lemma}

\begin{proof}
Recall that $(\xi_{n},S_{n})_{n\ge 0}$ with $S_{n}=\log|\Lambda_{n}\cdots\Lambda_{1}(\xi_{0})|$ constitutes a \MRW, which is nonarithmetic under the hypotheses of Thm.~\ref{thm:main case 1}. Put
\begin{gather*}
N(t)\,:=\,\inf\{n\ge 1:S_{n}>t\},\quad R(t)\,:=\,(S_{N(t)}-t)\1_{\{N(t)<\infty\}}
\shortintertext{and}
Z(t)\,:=\,\xi_{N(t)}\1_{\{N(t)<\infty\}}.
\end{gather*}
We first show that
\begin{equation}\label{eq:d3}
C(\delta)\,:=\,\lim_{t\to\infty}t^{\kappa}\,\Prob\!\left[\sup_{n\ge 1}|\Lambda_{n}\cdots\Lambda_{1}(\delta)|>t\right]
\end{equation} 
is strictly positive for each $\delta\in\bbS^{0}$.
This implies \eqref{eq3:aux1 pos const} when combined with the fact that $\Lle_{n}\eqdist\Lri_{n}$ for each $n\ge 1$, where
\begin{equation}\label{eq:lle}
\Lle_{n}\,:=\,\Lip(\Lambda_{n}\cdots\Lambda_{1})
\end{equation}
for $n\in\N$ and $\Lle_{0}:=1$. On the other hand, it does not directly imply \eqref{eq1:aux1 pos const} because the equality in law holds for fixed $n$ only and the main difficulty in our proof is indeed to reverse the order of random iterations.

\vspace{.1cm}
To prove our claim, define $f(\delta,t):=e^{-\kappa t}/v_{\delta}(\kappa)\,\1_{\IRge}(t)$ for $(\delta,t)\in\bbS^{0}\times\R$. Then
\begin{align*}
\frac{e^{\kappa t}}{v_{\delta}(\kappa)}\,&\Prob\!\left[\sup_{n\ge 1}|\Lambda_{n}\cdots\Lambda_{1}(\delta)|>e^{t}\right]\ =\ \frac{e^{\kappa t}}{v_{\delta}(\kappa)}\,\Prob_{\delta}\!\left[\sup_{n\ge 1}S_{n}>t\right]\\
&=\ \frac{e^{\kappa t}}{v_{\delta}(\kappa)}\,\Prob_{\delta}[N(t)<\infty]\\
&=\ \sum_{n\ge 1}\Erw_{\delta}^{(\kappa)}\!\left[\1_{\{N(t)=n\}}e^{-\kappa(S_{n}-t)}/ v_{\xi_{n}}(\kappa)\right]\\
&=\ \sum_{n\ge 1}\Erw_{\delta}^{(\kappa)}\!\left[\1_{\{N(t)=n\}}f(\xi_{n},S_{n}-t)\right]\\
&=\ \Erw_{\delta}^{(\kappa)}\!\left[\1_{\{N(t)<\infty\}}f(Z(t),R(t))\right],
\end{align*}
and the last expectation converges to a positive limit by an extension of the Markov Renewal  Lemma \ref{lem:RThm case 1}, see \cite[Thm.~1]{Kesten:74} and \cite[Cor.~3.2]{Alsmeyer:97}.

\vspace{.1cm}
Turning to the proof of \eqref{eq1:aux1 pos const}, we  fix $m\in\N$, define the stopping time
$$ \tau_{m,t}\,:=\,\inf\{  n\ge 0:\; t < \Lri_{n} < mt\} $$
and point out that
\begin{equation}\label{eq:star}
\begin{split}
\sum_{n\ge 1}\,&\Prob[t<\Lri_{n}\le mt]\\
&=\ \sum_{n\ge 0}\Prob[\tau_{m,t}<\infty,\,t<\Lri_{\tau_{m,t}+n}\le mt]\\
&\le\ \sum_{n\ge 0}\Prob[\tau_{m,t}<\infty,\,\Lri_{\tau_{m,t}}\cdot\Lip(\Lambda_{\tau_{m,t}+1}\cdots\Lambda_{\tau_{m,t}+n})>t]\\
&\le\ \sum_{n\ge 0}\Prob[\tau_{m,t}<\infty,\,\Lip(\Lambda_{\tau_{m,t}+1}\cdots\Lambda_{\tau_{m,t}+n})>1/m]\\
&\le\ \Prob[\tau_{m,t}<\infty]\sum_{n\ge 1}\Prob[\Lri_{n}>1/m].
\end{split}
\end{equation}
Since $\rho(\vth)<1$ for some $\vth>0$, Lemma \ref{lem:Ptheta vs Lambda} ensures the existence of constants $\rho<1$ and $K<\infty$ such that
\begin{equation}\label{eq:mean Lip power n}
\Erw\big[{\Lri_{n}}^{\vth}\big]\,\le\,K \rho^{n}
\end{equation}
for all $n\in \N$, which in turn implies that
$$ \beta\,:=\,\sum_{n\ge 1}\Prob[\Lri_{n}>1/m] $$
is finite.
Then \eqref{eq:star} yields
\begin{align*}
\Prob\bigg[\sup_{n\ge 1}\Lri_{n}>t\bigg]\ &\ge\ \Prob[\tau_{m,t}<\infty]\\
&\ge\ \beta^{-1}\,\sum_{n\ge 1}\Prob[t<\Lri_{n}\le mt].
\end{align*}
Next, use $\Lri_{n}\eqdist\Lle_{n}$ and
\begin{equation}\label{eq:def of Lip(Lambda_{n})}
\Lle_{n}\ =\ |\Lambda_{n}\cdots\Lambda_{1}(1)|\vee|\Lambda_{n}\cdots\Lambda_{1}(-1)|
\end{equation}
for each $n\in\N$ to obtain
\begin{align*}
\sum_{n\ge 1}\,&\Prob[t<\Lri_{n}\le mt]\ =\ \sum_{n\ge 1}\,\Prob[t<\Lle_{n}\le mt]\\
&\ge\ \sum_{n\ge 1}\Prob[t<|\Lambda_{n}\cdots\Lambda_{1}(1)|\le mt,|\Lambda_{n}\cdots\Lambda_{1}(-1)|\le mt]\\
&\ge\ \sum_{n\ge 1}\Big(\Prob[t<|\Lambda_{n}\cdots\Lambda_{1}(1)|\le mt]-\Prob[|\Lambda_{n}\cdots\Lambda_{1}(-1)|>mt]\Big)\\
&=\ \sum_{n\ge 1}\Big(\IPplus[0<S_{n}-\log t\le\log m]-\IPminus[S_{n}-\log t>\log m]\Big)\\
&=:~I_{1}(t)+I_{2}(t).
\end{align*}
Moreover, with $v^{*}(\kappa):=\vminus(\kappa)\vee\vplus(\kappa)>0$,
\begin{gather*}
I_{1}(t)\ \ge\ \frac{t^{-\kappa}\vplus(\kappa)}{v^{*}(\kappa)}\sum_{n\ge 1}\Eplus^{(\kappa)}e^{\kappa(\log t-S_{n})}\1_{[0,\log m)}(S_{n}-\log t)
\shortintertext{and, similarly,}
I_{2}(t)\ \ge\ \frac{t^{-\kappa}\vminus(\kappa)}{v^{*}(\kappa)}\sum_{n\ge 1}\Eminus^{(\kappa)}e^{\kappa(\log t-S_{n})}\1_{(\log m,\infty)}(S_{n}-\log t)
\end{gather*}
By another appeal to Lemma \ref{lem:RThm case 1},
\begin{gather*}
\lim_{t\to\infty}t^{\kappa}I_{1}(t)\ \ge\ \frac{\vplus(\kappa)}{v^{*}(\kappa)\rho'(\kappa)}\int_{0}^{\log m}e^{-\kappa x}\ dx\ =\ \frac{\vplus(\kappa)(1-m^{-\kappa})}{\kappa v^{*}(\kappa)\rho'(\kappa)}
\shortintertext{and}
\lim_{t\to\infty}t^{\kappa}I_{2}(t)\ \ge\ \frac{\vminus(\kappa)}{v^{*}(\kappa)\rho'(\kappa)}\int_{\log m}^{\infty}e^{-\kappa x}\ dx\ =\ \frac{\vminus(\kappa)m^{-\kappa}}{\kappa v^{*}(\kappa)\rho'(\kappa)}.
\end{gather*}
By putting the previous estimates together and fixing $m$ sufficiently large, we see that,
\begin{align*}
\liminf_{t\to\infty}t^{\kappa}\,\Prob\left[\sup_{n\ge 1}\Lri_{n}>t\right]\ \ge\ \liminf_{t\to\infty}\frac{t^{\kappa}}{\beta}(I_{1}(t)+I_{2}(t))\ >\ 0,
\end{align*}
and thus \eqref{eq1:aux1 pos const} holds true.

\vspace{.2cm}
Left with the proof of \eqref{eq2:aux1 pos const}, we first verify the weaker assertion
\begin{align}\label{eq3:aux1 pos const2}
t^{\kappa}\,\Prob\left[\sup_{n\ge 1}|\Lambda_{1}\cdots\Lambda_{n}(\delta)|>t\right]\ =\ t^{\kappa}\,\Prob_{\delta}\left[\sup_{n\ge 1}|\wh{Z}_{n}|>t\right]\ \ge\ K_{1}
\end{align}
for some $K_{1}>0$, each $\delta\in\bbS^{0}$ and all $t\ge 1$. It is enough to consider $\delta=+1$. By irreducibility, we can fix $\eta\in (0,1)$ sufficiently small such that
\begin{gather*}
p:\,=\,\Prob(\Aplus<-\eta)\wedge\Prob(\Aminus>\eta)\,>\,0.
\shortintertext{Next, put}
\tau\,=\,\tau_{t/\eta}\,:=\,\inf\{n\ge 1 : \Lri_{n}>t/\eta\}
\shortintertext{with associated events}
B_{+}\,:=\,\{\Lri_{\tau}
=|\Lambda_{1}\cdots\Lambda_{\tau}(+1)|,\tau<\infty\},\\
B_{-}\,:=\,\{\Lri_{\tau}
=|\Lambda_{1}\cdots\Lambda_{\tau}(-1)|,\tau<\infty\}.
\end{gather*}
Notice that $|\Lambda_{1}\cdots\Lambda_{\tau}(+1)|>t/\eta>t$ on $B_{+}$,
\begin{align*}
|\Lambda_{1}\cdots\Lambda_{\tau+1}(+1)|\ =\ |\Lambda_{1}\cdots\Lambda_{\tau}(-1)||\Lambda_{\tau+1}(+1)|\ >\ \frac{t}{\eta}\cdot\eta\ =\ t
\end{align*}
on $B_{-}\cap\{\Lambda_{\tau+1}(+1)<-\eta\}$, and that $\Lambda_{\tau+1}(+1)$ is independent of $B_{-}$ with the same law as $\Lambda_{1}(+1)=\Aplus_{1}$. Then it follows that
\begin{align*}
\Prob&\left[\sup_{n\ge 1}|\Lambda_{1}\cdots\Lambda_{n}(+1)|>t\right]\ \ge\ \Prob\left[\tau<\infty,\sup_{n\ge 1}|\Lambda_{1}\cdots\Lambda_{n}(+1)|>t\right]\\
&\ge\ \Prob[B_{+}]\,+\,\Prob[B_{-}\cap\{\Lambda_{\tau+1}(+1)<-\eta\}]\\
&=\ \Prob[B_{+}]\,+\,\Prob[B_{-}]\,\Prob[\Lambda_{\tau+1}(+1)<-\eta]\\
&\ge\ p\big(\Prob[B_{+}]\,+\,\Prob[B_{-}]\big)\ \ge\ p\,\Prob[\tau<\infty]\ \ge\ \frac{p\eta^{\kappa}}{Kt^{\kappa}}\ =\ \frac{K_{1}}{t^{\kappa}}
\end{align*}
for all $t\ge 1$ which is the desired result.

\vspace{.2cm}
Finally, we must prove that \eqref{eq3:aux1 pos const2} does indeed already imply \eqref{eq2:aux1 pos const}. To this end, we note that
\begin{gather*}
\begin{split}
\Prob\bigg[\sup_{n\ge 1}&\,\Lambda_{1}\cdots\Lambda_{n}(+1)>t\bigg]\\
&=\ \Prob\big[\exists\,n\ge 1:\Lambda_{1}\cdots\Lambda_{n}(+1)>0,\,
|\Lambda_{1}\cdots\Lambda_{n}(+1)|>t\big]\\&
\ge\ \Prob\big[\exists\,n\ge 1:\Lambda_{1}\cdots\Lambda_{n}(+1)>0,\,|\Lambda_{1}\cdots\Lambda_{n}(+1)|>t/\eta\big]
\end{split}
\shortintertext{and}
\begin{split}
\Prob&\bigg[\sup_{n\ge 1}\Lambda_{1}\cdots\Lambda_{n}(+1)>t\bigg]\\
&\ge\ \Prob\bigg[\Lambda_{1}(-1)>\eta,\,\sup_{n\ge 2}\Lambda_{1}\cdots\Lambda_{n}(+1)>t\bigg]\\
&\ge\ \Prob\big[\Aminus_{1}>\eta,\exists\,n\ge 2:\Lambda_{2}\cdots\Lambda_{n}(+1)<0,\Aminus_{1}|\Lambda_2\cdots\Lambda_{n}(+1)|>t\big]\\
&\ge\ p\,\Prob\big[\exists\,n\ge 1: \Lambda_{1}\cdots\Lambda_{n}(+1)<0,\,|\Lambda_{1}\cdots\Lambda_{n}(+1)|>t/\eta\big]
\end{split}
\end{gather*}
Hence, assuming \eqref{eq3:aux1 pos const2}, a combination of both yields
\begin{align*}
2\,\Prob&\bigg[\sup_{n\ge 1}\Lambda_{1}\cdots\Lambda_{n}(+1)>t\bigg]\\
&\ge\ \Prob\big[\exists\,n\ge 1:\Lambda_{1}\cdots\Lambda_{n}(+1)>0,\,|\Lambda_{1}\cdots\Lambda_{n}(+1)|>t/\eta\big]\\
&\hspace{1cm}+\ p\,\Prob\big[\exists\,n\ge 1:\Lambda_{1}\cdots\Lambda_{n}(+1)<0,\,|\Lambda_{1}\cdots\Lambda_{n}(+1)|>t/\eta\big]\\
&\ge\ p\,\Prob\left[\sup_{n\ge 1}|\Lambda_{1}\cdots\Lambda_{n}(+1)|>t/\eta\right]\ \ge\ pK_{1}\eta^{\kappa}t^{-\kappa}
\end{align*}
for all $t\ge 1$ as claimed.
\qed \end{proof}

\begin{Lemma}\label{lem:aux2 pos const}
Under the hypotheses of Prop.~\ref{prop:tail constants positive}, Condition \eqref{eq2:lower bound t^kappa nu} holds.
\end{Lemma}

\begin{proof}
Recall that $\wh{Y}_{\infty}=\sum_{n\ge 0}\Lri_{n}B_{n+1}$. For \eqref{eq2:lower bound t^kappa nu}, it therefore suffices to verify
\begin{equation}\label{eq:pt2x}
\Prob\Big[\max_{n\ge 0}\Lri_{n}B_{n+1} > t\Big]\ \le\ \frac{K}{t^{\kappa}}.
\end{equation}
and
\begin{equation}\label{eq:pt2}
\Prob\Big[\wh{Y}_{\infty}>t,\,\max_{n\ge 0}\Lri_{n}B_{n+1}\le t\Big]\ \le\ \frac{K}{t^{\kappa}}.
\end{equation}
Here and in the following $K$ denotes a generic positive constant that may differ from line to line. To prove \eqref{eq:pt2x}, we recall \eqref{eq:lle} and define the two events
\begin{equation}\label{eq:d31}
V_{n}^{i}\,=\,\big\{e^{i}t<\Lri_{n}B_{n+1}\le e^{i+1}t\big\},\quad U_{n}^{i}\,\,=\big\{e^{i}t <\Lle_{n}B_{0} \le  e^{i+1} t  \big\}.
\end{equation}
of equal probability for all $i,n\in\N_{0}$.

\vspace{.1cm}
Fix $\vth>0$ and $\rho<1$ as in inequality \eqref{eq:mean Lip power n} and choose $M\in\N$ large enough such that $2e^{\vth}K\rho^M<1-\rho^M$. Then
\begin{align*}
\Prob&\Big[\max_{n\ge 0}\Lri_{n}B_{n+1} > t\Big]\\
&=\ \sum_{i\ge 0} \Prob\Big[e^{i}t<\max_{n\ge 1}\Lri_{n-1}B_{n}\le e^{i+1}t\Big]
\ =\ \sum_{i\ge 0}\Prob\bigg[\bigcup_{n\ge 0}V_{n}^{i}\bigg] \\
&=\ \sum_{i\ge 0}\Prob\bigg[\bigcup_{m=0}^{M-1}\bigcup_{n\ge 0}V_{nM+m}^{i}\bigg]
\ \le\ \sum_{m=0}^{M-1}\sum_{i\ge 0}\Prob\bigg[\bigcup_{n\ge 0}V_{nM+m}^{i}\bigg]\\
&\le\ \sum_{m=0}^{M-1}\sum_{i\ge 0}\sum_{n\ge 0}\Prob\big[V_{nM+m}^{i}\big]
\ =\ \sum_{m=0}^{M-1} \sum_{i\ge 0} \sum_{n\ge 0}\Prob\big[U_{nM+m}^{i}\big],
\end{align*} 
whence \eqref{eq:pt2x} follows if we prove that
\begin{equation}\label{eq:d1}
\sum_{i\ge 0}\sum_{n\ge 0}\Prob\big[ U_{nM+m}^{i}\big]\ \le\ \frac{K}{t^{\kappa}}
\end{equation} 
for $m=0,\ldots,M-1$. We confine ourselves to the case $m=0$ and note first that
\begin{equation}\label{eq:d2}
\sum_{n\ge 0}\Prob\big[U_{nM}^{i}\big]\ \le\ \Prob\bigg[\bigcup_{n\ge 0}U_{nM}^{i}\bigg]\,+\,\sum_{n\ge 0}\sum_{j>n}\Prob[U_{nM}^{i}\cap U_{jM}^{i}]
\end{equation}
for any $i$. Then, by using how $\rho$ and $M$ have been chosen, we find for any fixed $n$ that
\begin{align*}
\sum_{j>n}\Prob[U_{nM}^{i}\cap U_{jM}^{i}]\ &\le\ \sum_{j>n}\Prob\big[\Lle_{nM}B_{0}\le e^{i+1}t,
\,\Lle_{jM}B_{0}>e^{i}t\big]\\
& \le\ \Prob[U_{nM}^{i}]\sum_{j>n}\Prob\big[\Lip(\Lambda_{jM}\cdots \Lambda_{nM+1}) > e^{-1}\big]\\
&\le\ \Prob[U_{nM}^{i}]\sum_{j>n}e^{\vth}K\rho^{(j-n)M}\ \le\ \frac{1}{2}\,\Prob[U_{nM}^{i}].
\end{align*}
which in combination with \eqref{eq:d2} leads to
$$ \sum_{n\ge 0}\Prob\big[ U_{nM}^{i}\big]\ \le\ 2\,\Prob\bigg[\bigcup_{n\ge 0}U_{nM}^{i}\bigg] $$
and then finally to
\begin{align*}
\sum_{i\ge 0} \sum_{n\ge 0} \Prob\big[ U_{nM}^{i}\big]\ &\le\ 2\sum_{i\ge 0} \Prob\bigg[\bigcup_{n\ge 0} U_{nM}^{i}\bigg]\ \le\ 2\sum_{i\ge 0} \Prob\bigg[\sup_{n\ge 0}\Lle_{n}>e^{i}t/B_{0}\bigg]\\
&\le\ \bigg(K\,\Erw B_{0}^{\kappa}\sum_{i\ge 0}e^{-i \kappa} \bigg) \cdot t^{-\kappa},
\end{align*}
where the penultimate inequality follows from the definition of the $U_{n}^{i}$ (see \eqref{eq:d31}) and the last one by \eqref{eq:d3} and the independence of $\sup_{n}\Lri_{n}$ and $B_{0}$. This completes the proof of \eqref{eq:d1} and thus also of \eqref{eq:pt2x}. 

\medskip

Turning to inequality \eqref{eq:pt2}, we define the family of events
$$ W_{j}\,:=\,\bigg\{n:\; \frac{t}{e^{j+1}}<\Lri_{n-1}B_{n}\le\frac{t}{e^{j}}\bigg\},\quad j\ge 0 $$
and claim that, for some $\rho\in (0,1)$ and all $j\in\N_{0}$,
\begin{equation*}
\Prob\big[\card(W_{j})>\ell\big]\,\le\,K\frac{\rho^{\ell}e^{\kappa j}}{t^\kappa},
\end{equation*}
where $\card$ denotes cardinality of a set. To verify this, pick $\vth\in (0,\kappa)$ and $\rho\in (0,1)$ in accordance with \eqref{eq:mean Lip power n} and observe that
\begin{equation*}
\Prob\bigg[\sup_{m\ge\ell}\Lri_{m-1}B_{m}\ge s\bigg]\,\le\,\sum_{m\ge\ell}\frac{\Erw{\Lri_{m-1}}^{\vth}\,\Erw B^{\vth}}{s^{\vth}}\,\le\,K\,\frac{\rho^{\ell}\,\Erw B^{\vth}}{s^{\vth}}.
\end{equation*}
Let $\tau_{i}=\tau_{i}(j)$ for $i=1,2$ be two  smallest elements of $W_{j}$, with $\tau_{1}:=\infty$ if $W_{j}$ is empty and $\tau_{2}:=\infty$ if $\card(W_{j})\le 1$. Put also $L_{k}^{k+m}:=\Lip(\Lambda_{k}\cdots\Lambda_{m+k})$. Then
\begin{align*}
&\Prob[\card(W_{j})>\ell+1]\\
&\le\ \Prob\bigg[\tau_{2}<\infty,\ \exists\,m\ge\ell: \frac{t}{e^{j+1}} < \Lri_{\tau_{2}+m-1}B_{\tau_{2}+m}<\frac{t}{e^{j}}\bigg]\\
&\le\ \Prob\bigg[\tau_{2}<\infty,\ \exists\,m\ge\ell:\Lri_{\tau_{2}-1}\Lip(\Lambda_{\tau_{2}})L^{\tau_{2}+m-1}_{\tau_{2}+1} B_{\tau_{2}+m}>\frac{t}{e^{j+1}}\bigg]\\
& \le\ \Prob\bigg[\tau_{2}<\infty,\ \exists\,m\ge\ell:L^{\tau_{2}+m-1}_{\tau_{2}+1} B_{\tau_{2}+m}>\frac{t }{e^{j+1}
\Pi_{\tau_{2}-1}B_{\tau_{2}}}\,\frac{B_{\tau_{2}}}{\Lip(\Lambda_{\tau_{2}})}\bigg]\\
&\le\ \Prob\bigg[\tau_{2}<\infty,\ \exists\,m\ge\ell:L^{\tau_{2}+m-1}_{\tau_{2}+1} B_{\tau_{2}+m}>\frac{B_{\tau_{2}}}{e\,\Lip(\Lambda_{\tau_{2}})}\bigg]\\
&\le\ Ke^{\vth}\rho^{\ell}\,\Erw(B^{\vth})\,\Erw\bigg[\1_{\{\tau_{2}<\infty\}}\frac{\Lip(\Lambda_{\tau_{2}})^{\vth}}{B_{\tau_{2}}^{\vth}}\bigg]
\qquad (\text{since } B_{\tau_{2}}\ge 1)
\\
&\le\ Ke^{\vth}\rho^{\ell}\,\Erw(B^{\vth})\,\Erw\big[\Lip(\Lambda_{{\tau_{2}}})^{\vth}\big] \Prob\big[\tau_1 <\infty\big]\\
&\le\ Ke^{\vth}\rho^{\ell}\,\Erw(B^{\vth})\,\Erw\big[\Lip(\Lambda_{{\tau_{2}}})^{\vth}\big] \Prob\bigg[\sup_{n\ge 1}\Lri_{n-1}B_{n}>\frac{t}{e^{j+1}}\bigg]\\
&\le\ K\rho^{\ell+1}  e^{\kappa j} t^{-\kappa},
\end{align*}
where the last inequality follows from \eqref{eq:d3}.

\vspace{.1cm}
Returning to the proof of \eqref{eq:pt2}, we note that the occurrence of $\wh{Y}_{\infty}>t$ and $\sup_{n\ge 1}\Lri_{n-1}B_{n}\le t$ entails that at least one $W_{j}$ must be relatively large, more precisely, that a.s.~$\card(W_{j})>e^{j}/2(j+1)^{2}$ for some $j\ge 0$. Indeed, if the latter fails, then
$$ \wh{Y}_{\infty}\ =\ \sum_{j\ge 0}\sum_{n\in W_{j}}\Lri_{n-1}B_{n}\ \le\ \sum_{j\ge 0}\card(W_{j})\cdot\frac t{e^{j}}\ \le\ \sum_{j\ge 0}\frac{e^{j}}{2(j+1)^{2}}\cdot\frac{t}{e^{j}}\ <\ t. $$
Hence, we finally arrive at
\begin{gather*}
\Prob\big[ \wh{Y}_{\infty}>t,\max\Pi_{n-1} B_{n}\le t\big]\ \le\ \sum_{j\ge 0}\Prob\bigg[\text{card} (W_{j})>\frac {e^{j}}{2(j+1)^{2}}\bigg]\\
\le\ K\sum_{j\ge 0}\frac{\rho^{e^{j}/2(j+1)^{2}}e^{\kappa j}}{t^\kappa}\ <\ Kt^{-\kappa}
\end{gather*}
and thus at the desired conclusion.
\qed \end{proof}

Proposition \ref{prop:tail constants positive} is a direct consequence of the Lemmata \ref{lem:lower bound t^kappa nu}, \ref{lem:aux1 pos const} and \ref{lem:aux2 pos const}.

\vspace{.3cm}

\textit{\bfseries Case 2 (unilateral case)}. $p_{-+}>0$ and $p_{+-}=0$.

\begin{Prop}\label{prop:tail constants positive2}
(a) If the hypotheses of Thm.~\ref{thm:main case 2}(a) and $\pmm'(0)<0$ hold, then the constant $\Cminus$ in \eqref{eq:left tails main case 2} is strictly positive for any stationary law $\nu$ of unbounded support at $-\infty$.

\vspace{.1cm}
(b) If the hypotheses of Thm.~\ref{thm:main case 2}(b) and $\ppp'(0)<0$ hold, then the constant $\Cplus$ in \eqref{eq:right tails main case 2A} is strictly positive for any stationary law $\nu$ of unbounded support at $+\infty$.

\vspace{.1cm}
(c) If the hypotheses of Thm.~\ref{thm:main case 2}(c) and $\pmm'(0)<0$ hold, then the constant $\Cminplus$ in \eqref{eq:right tails main case 2B} is strictly positive for any stationary law $\nu$ of unbounded support at both $-\infty$ and $+\infty$..
\end{Prop}

\textit{\bfseries Case 3 (separated case)}. $\pmp=0$ and $\ppm=0$.

\begin{Prop}\label{prop:tail constants positive3}
(a) If the hypotheses of Thm.~\ref{thm:main case 3}(a) and $\pmm'(0)<0$ hold, then the constant $\Cminus$  is strictly positive for any stationary law $\nu$ of unbounded support at $-\infty$.

\vspace{.1cm}
(b) If the hypotheses of Thm.~\ref{thm:main case 3}(b) and $\ppp'(0)<0$, then the constant $\Cplus$  is strictly positive for any stationary law $\nu$ of unbounded support at $+\infty$.
\end{Prop}

\section{Existence of a stationary distribution}
\label{sec:existence}

In order to wrap up our presentation, this very short section provides conditions which ensure the existence of at least one stationary distribution of the given \ALIFS~and are directly seen to hold in our main results. We do not strive for utmost generality here, nor do we address the uniqueness question. While existence of a stationary distribution depends on the behavior of the IFS at infinity and could be derived from weaker assumptions, uniqueness is a "local" property and needs "local" assumptions that are not imposed in the very general setting of this work. On the other hand, the subsequent result is tailored to our needs and very easily deduced by a tightness argument using \eqref{eq:theta-norm whYinfty}.

\begin{Prop}\label{prop:stationary distribution}
Suppose that there exists $\vth>0$ such that $\rho(\vth)<1$ and $\Erw B^{\vth} <\infty$.
Then the \ALIFS\ $(X_{n})_{n\ge 0}$ admits at least one stationary distribution.
\end{Prop}

As a particular consequence, the convexity of the spectral radius $\rho(\theta)$ yields the existence of an invariant law whenever \eqref{eq:moments AlogA B} holds with $\kappa$ such that $\rho(\kappa)=1$.

\begin{proof}
Suppose that $(X_{n})_{n\ge 0}$ has initial state $X_0=0$ and recall from Lemma
\ref{lem:comparison with LIFS} that
\begin{align*}
|X_{n}|\ =\ |X_{n}-\Lambda_{n}\cdots\Lambda_{1}(0)|\ \le\ Y_{n}
\end{align*}
for each $n\in\N$. Using also $Y_{n}\eqdist\wh{Y}_{n}\uparrow\wh{Y}_{\infty}$, we infer that
$$ \Prob_{0}[|X_{n}|>K]\ \le\ \Prob[Y_{n}>K]\ =\ \Prob[\wh{Y}_{n}>K]\ \le\ \Prob[\wh{Y}_{\infty}>K] $$
for each $K>0$ and thus the uniform tightness of $(X_{n})_{n\ge 0}$ under $\Prob_{0}$ because, by \eqref{eq:theta-norm whYinfty}, $\wh{Y}_{\infty}$ is almost surely finite under the given assumptions. Since $(X_{n})_{n\ge 0}$ is also a Feller chain, the existence of a stationary distribution now follows by the Krylov-Bogoliubov theorem, see e.g. \cite[Thm.~3.1.1]{DaPratoZab:96}.
\qed \end{proof}

\section{The \AR(1) model with \ARCH(1) errors revisited}\label{sec:AR(1) with ARCH errors}

This model has already been mentioned in Subsection \ref{subsec:examples}. Defined as  the \ALIFS~generated by i.i.d.~copies of the random function
$$ \Psi(x)\,=\,\alpha x+Z\!\left(\beta+\lambda x^{2}\right)^{1/2} $$
for some $(\alpha,\beta,\lambda)\in\R\times\IRg^{2}$ and a random variable $Z$, it provides an ideal example to illustrate our results because all three cases can occur depending on how the parameters $\alpha,\beta,\lambda$ and (the range of) the random variable $Z$ are chosen. Since
$$ 0\ \le\ \left(\beta+\lambda x^{2}\right)^{1/2}-\left(\lambda x^{2}\right)^{1/2}\ =\ \frac{\beta}{\left(\beta+\lambda x^{2}\right)^{1/2}+\left(\lambda x^{2}\right)^{1/2}}\ \le\ \beta^{1/2} $$
for all $x\in\R$, we see that Condition \eqref{eq:def2 AL} holds with
\begin{gather*}
\Apm\,=\,\alpha\pm\lambda^{1/2}Z\quad\text{and}\quad B\,=\,\beta^{1/2}Z,
\shortintertext{so that}
\pmp\,=\,\Prob[Z>\alpha/\lambda^{1/2}]\quad\text{and}\quad\ppm\,=\,\Prob[Z<-\alpha/\lambda^{1/2}].
\shortintertext{and}
P(\theta)\ =\ \begin{pmatrix}
\Erw\big[|\alpha-\lambda^{1/2} Z|^{\theta}\1_{\{Z<\alpha/\lambda^{1/2}\}}\big] &\Erw\big[|\alpha-\lambda^{1/2} Z|^{\theta}\1_{\{Z>\alpha/\lambda^{1/2}\}}\big]\\[1mm]
\Erw\big[|\alpha+\lambda^{1/2} Z|^{\theta}\1_{\{Z<-\alpha/\lambda^{1/2}\}}\big] &\Erw\big[|\alpha+\lambda^{1/2} Z|^{\theta}\1_{\{Z>-\alpha/\lambda^{1/2}\}}\big]
\end{pmatrix}.
\end{gather*}
Now one can easily see that all three cases can occur, namely the
\begin{itemize}\itemsep1pt
\item \emph{irreducible case} if $\Prob[Z>\alpha/\lambda^{1/2}]$ and $\Prob[Z<-\alpha/\lambda^{1/2}]$ are both positive,
\item \emph{unilateral case} if $Z>-\alpha/\lambda^{1/2}$ a.s.~and $\Prob[Z>\alpha/\lambda^{1/2}]>0$, and
\item \emph{separated case} if $|Z|\le\alpha/\lambda^{1/2}$ a.s.
\end{itemize}
By invoking our results, we conclude under the respective additional conditions imposed there, especially (in all three cases)
\begin{equation*}
\Erw|Z|^{\kappa}\log|Z|\,<\,\infty,
\end{equation*}
that any stationary law of unbounded support has
\begin{itemize}\itemsep1pt
\item \emph{irreducible case:} left and right power tails of order $\kappa$ with $\kappa$ defined as the minimal positive value such that $\rho(\kappa)=1$ and with constants $\Cminus,\Cplus>0$ in \eqref{eq:tails main case 1}.
\item \emph{unilateral case:} has left power tails of order $\kappaminus$ and/or right power tails of order $\kappaminus\wedge\kappaplus$ with $\kappaminus,\kappaplus$ being the unique positive numbers (if they exist and are distinct) such that
$$ \Erw|\alpha-\lambda^{1/2} Z|^{\kappaminus}\1_{\{Z<\alpha/\lambda^{1/2}\}}\,=\,1\quad\text{and}\quad\Erw|\alpha+\lambda^{1/2} Z|^{\kappaminus}\1_{\{Z>-\alpha/\lambda^{1/2}\}}\,=\,1 $$
and with constants $\Cminus,\Cplus,\Cminplus>0$ in \eqref{eq:KRT1 main case 2}, \eqref{eq:right tails main case 2A} and \eqref{eq:right tails main case 2B}, respectively.
\item \emph{separated case}: has left power tails of order $\kappaminus$ and/or right power tails of order $\kappaplus$ with $\kappaminus,\kappaplus$ as in the previous case (if they exist) and with constants $\Cminus,\Cplus>0$ in \eqref{eq1:main case 3} and \eqref{eq2:main case 3}, respectively.
\end{itemize}
The case when $Z$ has a symmetric law, which rules out the unilateral case, has already been studied in \cite[Sect.~8.4]{EmKlMi:97} for $\alpha=0$ and Gaussian $Z$, in \cite{BorkovecKl:01},  and in \cite[Subsect.~6.1]{Alsmeyer:16} by showing that a stationary law must be symmetric as well and thus have left and right tails of the same order which in fact allows to resort to Goldie's implicit renewal theory.

\section{Appendix}

The following lemma confirms that $\rho(\theta)=1$ in a right neighborhood of 0 can occur in the irreducible case only if the nonlattice assumption \eqref{eq:log A nonlattice} is violated.

\begin{Lemma}\label{lem:rho(s)<1}
Suppose that $P(\theta)$ exists and has spectral radius $\rho(\theta)=1$ for all $\theta\in I=[0,\theta_{0}]$, $\theta_{0}>0$. Then one of the following alternatives holds:
\begin{itemize}[leftmargin=.8cm]\itemsep2pt
\item[(a)] $\pmp\wedge\ppm=0$ and thus $\Aminus=1$ or $\Aplus=1$ a.s.
\item[(b)] $\pmp(\theta)\ppm(\theta)\equiv\gamma>0$ for all $\theta\in I$ and
$$ \Aminus\ =\ \begin{cases} 1&\text{if }\Aminus>0\\
-a^{-1}&\text{if }\Aminus<0\end{cases}
\quad\text{and}\quad
\Aplus\ =\ \begin{cases} 1&\text{if }\Aplus>0\\
a&\text{if }\Aplus<0\end{cases}
\quad\text{a.s.} $$
for some $a>0$.
\end{itemize}
\end{Lemma}

Regarding \eqref{eq:log A nonlattice}, Alternative (b) indeed implies that it fails because
$$ \Prob_{\wh{\pi}(\kappa)}[\log|{}^{\xi_{0}}\!A|-a_{\xi_{1}}+a_{\xi_{0}}\in d\Z]\ =\ 1 $$
when choosing $a_{\pm}=0$ and $d=\log a$.

\begin{proof}
Using Formula \eqref{eq:rho(theta) explicit} for $\rho(\theta)$, one can readily check that $\rho(\theta)=1$ for all $\theta\in I$ holds iff
\begin{equation}\label{crucial identity}
(1-\pmm(\theta))(1-\ppp(\theta))\ =\ \pmp(\theta)\ppm(\theta)\quad\text{for all }\theta\in I.
\end{equation}
Assuming $\pmp\wedge\ppm>0$ and thus $\pmm\vee\ppp<1$, we infer $\pmm(\theta)\vee\ppp(\theta)<1$ for all $\theta\in I'=[0,\theta_{1}]\subseteq I$ for some $\theta_{1}>0$, w.l.o.g.~$I'=I$. Observe also that $\pmp(\theta)=\pmp\,\Erw[|\Aminus|^{\theta}|\Aminus<0]$ is a moment generating function modulo the scalar $\pmp$ and therefore log-convex on $I$. The same holds naturally for $\ppm(\theta)$. On the other hand, the functions
$$ \log(1-\pmm(\theta))\quad\text{and}\quad\log(1-\ppp(\theta)) $$
are concave, being compositions of an increasing concave function with a concave function. Consequently, the logarithms of the products in \eqref{crucial identity} are both concave and convex and thus linear on $I$. This shows that, with $\gamma:=\pmp\ppm$,
\begin{equation}\label{eq:mgf identities}
\pmp(\theta)\ppm(\theta)\ =\ \gamma e^{b\theta}\quad\text{and}\quad\frac{1}{(1-\pmm(\theta))(1-\ppp(\theta))}\ =\ \gamma^{-1}e^{-b\theta}
\end{equation}
for all $\theta\in I$ and some $b\in\R$. Since $(1-\pmm(\theta))^{-1}$ and $(1-\ppp(\theta))^{-1}$ are the moment generating functions of the defective renewal measures
\begin{gather*}
\IHminus\ =\ \delta_{0}\,+\,\sum_{n\ge 1}\pmm^{n}\,\Prob[\log|\Aminus|\in\cdot|\Aminus>0]^{*n}
\shortintertext{and}
\IHplus\ =\ \delta_{0}\,+\,\sum_{n\ge 1}\ppp^{n}\,\Prob[\log|\Aplus|\in\cdot|\Aplus>0]^{*n},
\end{gather*}
respectively, where $\delta_{0}$ denotes Dirac measure at 0, we infer that $\IHminus*\IHplus$ puts all mass at $-b$ which is only possible if $b=0$ and
$$ \Prob[\Aminus=1|\Aminus>0]\ =\ \Prob[\Aplus=1|\Aplus>0]\ =\ 1. $$
Now the first identity of \eqref{eq:mgf identities} provides that $\gamma^{-1}\pmp(\theta)\ppm(\theta)$ is the moment generating function of both $\delta_{0}$ and of two independent random variables with respective laws $\Prob[\log|\Aminus|\in\cdot|\Aminus<0]$ and $\Prob[\log|\Aplus|\in\cdot|\Aplus<0]$, giving
$$ \Prob[\Aminus=-a^{-1}\in\cdot|\Aminus<0]\ =\ \Prob[\Aplus=a|\Aplus<0]\ =\ 1\quad\text{for some }a>0. $$
This completes the proof.
\qed \end{proof}

\vspace{1cm}
\footnotesize
\noindent   {\bf Acknowledgements.}
The authors would like to express their sincere gratitude to an anonymous referee whose numerous suggestions and constructive comments helped to improve the final version of this article.\\
G. Alsmeyer was partially funded by the Deutsche Forschungsgemeinschaft (DFG) under Germany's Excellence Strategy EXC 2044--390685587, Mathematics M\"unster: Dynamics--Geometry--Structure.\\
D. Buraczewski was partially supported by the
National Science Center, Poland (grant number 2019/33/B/ST1/00207).

\bibliographystyle{abbrv}
\bibliography{ifs}

\end{document}